\documentclass[a4paper,reqno, 11pt]{amsart}
\usepackage{amsmath,amssymb,amsfonts,amsthm, mathrsfs}
\usepackage{lmodern}
\usepackage{lmodern}
\usepackage{makecell}
\usepackage{diagbox}
\usepackage{multirow}
\usepackage{booktabs}
\usepackage{microtype}
\usepackage{color,enumitem,graphicx}
\usepackage[colorlinks=true,urlcolor=blue, citecolor=red,linkcolor=blue,linktocpage,pdfpagelabels, bookmarksnumbered,bookmarksopen]{hyperref}
\usepackage[english]{babel}
\numberwithin{equation}{section}
\newtheorem{theorem}{Theorem}[section]
\newtheorem{Conjection}{Conjection}[section]
\newtheorem{lemma}[theorem]{Lemma}
\newtheorem{remark}[theorem]{Remark}
\newtheorem{definition}[theorem]{Definition}
\newtheorem{proposition}[theorem]{Proposition}
\newtheorem{corollary}[theorem]{Corollary}

\usepackage[a4paper, left = 3.2cm, right = 3.2cm, top = 2.54cm, bottom = 2.54cm]{geometry}
\begin{document}

	\title[4d energy supcritical INLS]
	{The defocusing energy-supercritical inhomogeneous NLS in four 
		space dimension }
	
	\author[X. Liu]{Xuan Liu}
	\address{School of Mathematics, Hangzhou Normal University, \ Hangzhou ,\ 311121, \ China} 
	\email{liuxuan95@hznu.edu.cn}
	
	\author[C. Xu]{Chengbin Xu}
	\address{School of Mathematics and Statistics, Qinghai Normal University, \ Xining, Qinghai,\  810008, \ China.} 
	\email{xcbsph@163.com}

	\begin{abstract}
		In this paper, we investigate the global well-posedness and scattering theory for the defocusing energy supcritical inhomogeneous nonlinear Schr\"odinger equation $iu_t + \Delta u =|x|^{-b} |u|^\alpha u$ in four space dimension, where $s_c := 2- \frac{2-b}{\alpha} \in (1, 2)$ and $0<b<\min \{ (s_c-1)^2+1,3-s_c\}$.
		We prove that if the solution has a prior bound in the critical Sobolev space, that is, $u \in L_t^\infty(I; \dot{H}_x^{s_c}(\mathbb{R}^4))$, then $u$ is global and scatters.   The proof of the main results is based on the concentration-compactness/rigidity framework developed by Kenig and Merle [Invent. Math. 166 (2006)], together with a long-time Strichartz estimate, a spatially localized Morawetz estimate, and a frequency-localized Morawetz estimate.

\vspace{0.3cm}

\end{abstract}
	
	\maketitle
	
	\medskip
	\section{Introduction}
	We consider the defocusing   inhomogeneous nonlinear Schrödinger equation (INLS) in $\mathbb{R}^4$:
	\begin{equation}
		\begin{cases}
			iu_t + \Delta u = |x|^{-b} |u|^{\alpha} u, \qquad (t,x) \in \mathbb{R} \times \mathbb{R}^4, \\
			u(0,x) = u_0(x),
		\end{cases}\label{INLS}
	\end{equation}
	where $\alpha > 0$, and $0 < b < 2$. 
	Note that the classical nonlinear Schr\"{o}dinger (NLS) equation corresponds to the case $b=0$. For the physical background and applications of the INLS model (\ref{INLS}) in nonlinear optics, see Gill \cite{Gill2000} and Liu-Tripathi \cite{Liu1994}.

For solutions $u$ that are sufficiently smooth and decay suitably at infinity, the energy 
	\begin{equation}
		E(u) = \int_{\mathbb{R}^4} \frac{1}{2} |\nabla u(t,x)|^2 + \frac{1}{\alpha + 2} |x|^{-b} |u(t,x)|^{\alpha + 2} \, dx \notag
	\end{equation}
	is conserved throughout the interval of existence. 
	The Cauchy problem \eqref{INLS} is scale-invariant. Specifically, the scaling transformation
	\[
	u(t,x) \longmapsto \lambda^{\frac{2-b}{\alpha}} u( \lambda^2 t,\lambda x), \quad \lambda > 0,
	\]
	leaves the class of solutions to \eqref{INLS} invariant. This transformation also identifies the critical space $\dot{H}^{s_c}(\mathbb{R}^4)$, where the critical regularity $s_c$ is defined as $s_c := 2 - \frac{2-b}{\alpha}$. 
	We classify \eqref{INLS} as mass-critical if $s_c = 0$, energy-critical if $s_c = 1$, inter-critical if $0 < s_c < 1$, and energy-supercritical if $s_c > 1$. 
 
The local well-posedness of \eqref{INLS} in $H^s(\mathbb{R} ^d)$, for $0 \le s \le \min\{1, \frac{d}{2}\}$, under suitable conditions on $b$ and $\alpha$, was first established by Guzman \cite{Guzman} using classical Strichartz estimates in Sobolev spaces. Subsequently, by employing techniques based on Lorentz spaces \cite{Aloui,AnKim2023ZAA} and weighted Strichartz estimates \cite{CCF,ChoHongLee,ChoLee,Lee2025arXiv}, the local theory in $H^s(\mathbb{R} ^d)$ was extended to the range $0 \le s < \min\{d, 1 + \frac{d}{2} \}$, under the condition $0 < \alpha$, $(d - 2s)\alpha \le 4 - 2b$, along with suitable assumptions on $b$.

In this paper, we focus on the global well-posedness and scattering theory for the Cauchy problem \eqref{INLS}. A wealth of results has been established in this direction: for the mass-critical case, see \cite{LiuMiaoZheng}; for the inter-critical case, see \cite{Campos,Dinh3,Dinh1,Dinh4,FG,FG2,MMZ,Murphy2022}; and for the energy-critical case, see \cite{ChoHongLee,ChoLee,GuzmanMurphy2021JDE,GX,LiuYangZhang}.
However, to the best of our knowledge, no results are currently available for the energy-supercritical regime. This work aims to address precisely that case.
 
 	\subsection{The nonlinear Schr\"odinger equation at critical regularity}

 To begin with, we review   global well-posedness and scattering theory  for the standard nonlinear Schrödinger equation (NLS) on  $ \mathbb{R} ^d:$
	\begin{equation}
		iu_t + \Delta u = \lambda |u|^\alpha u, \qquad (t,x) \in \mathbb{R} \times \mathbb{R}^d, \label{NLS0}
	\end{equation}
	which has seen significant progress in recent years.
	Due to conserved quantities at critical regularity, mass- and energy-critical cases have received the most attention. For the defocusing ($\lambda=+1$) energy-critical NLS, it is now known that initial data in $\dot{H}^1(\mathbb{R} ^d)$ leads to global solutions that scatter. This result was established first for radial data by Bourgain \cite{Bourgain1999}, Grillakis \cite{Grillakis2000}, Tao \cite{Tao2005}, and later for general data by Colliander et al. \cite{Colliander2008}, Ryckman-Visan \cite{RyckmanVisan2007}, and Visan \cite{Visan2007,Visan2012}. For the focusing ($\lambda=-1$) case, see \cite{Dodson2019ASENS,KenigMerle2006,KillipVisan2010}.
	
	For mass-critical NLS, it has been shown that arbitrary data in $L^2(\mathbb{R} ^d)$ leads to global solutions that scatter, first for radial data in $d \geq 2$ (see \cite{TaoVisanZhang2007,KillipTaoVisan2009,KillipVisanZhang2008}) and later for general data in all dimensions by Dodson \cite{Dodson2012,Dodson2015,Dodson2016a,Dodson2016b}.

	For cases where $s_c \neq 0, 1$, the energy- and mass-critical methods do not apply due to the absence of conserved quantities controlling the time growth of the $\dot{H}^{s_c}(\mathbb{R} ^d)$ norm of the solutions. However, it is conjectured that under \textit{a priori} control of a critical norm, global well-posedness and scattering hold for all $s_c > 0$ in any spatial dimension:
	\begin{Conjection}\label{CNLS0}
		Let $d \geq 1$, $\alpha \geq \frac{4}{d}$, and $s_c = \frac{d}{2} - \frac{2}{\alpha}$. Assume $u: I \times \mathbb{R}^d \to \mathbb{C}$ is a maximal-lifespan solution to \eqref{NLS0} such that 
		\[
		u \in L_t^\infty \dot{H}_x^{s_c}(I \times \mathbb{R}^d).
		\]
		Then $u$ is global and scatters as $t \to \pm \infty$.
	\end{Conjection}
	
	The first progress on Conjecture \ref{CNLS0} was made by Kenig and Merle \cite{KenigMerle2010}, who addressed the case $d = 3$ and $s_c = \frac{1}{2}$ using their concentration-compactness framework \cite{KenigMerle2006} and a scaling-critical Lin-Strauss Morawetz inequality. Subsequently,  
	In the inter-critical regime ($0 < s_c < 1$), some significant progress toward Conjecture \ref{CNLS0} has been achieved by Murphy \cite{Murphy2014,Murphy2014b,Murphy2015}, Xie-Fang \cite{XieFang2013} ,Gao-Miao-Yang \cite{GaoMiaoYang2019}, Gao-Zhao \cite{GaoZhao2019} and Yu \cite{Yu2021}.
	
	For energy-supercritical cases ($s_c > 1$), Killip and Visan \cite{KillipVisan2010} were the first to prove Conjecture \ref{CNLS0} for $d \geq 5$ under certain conditions on $s_c$. Murphy \cite{Murphy2015} later addressed the conjecture for radial initial data in $d = 3$ with $s_c \in (1, \frac{3}{2})$. By introducing long-time Strichartz estimates in the energy-supercritical setting, Miao-Murphy-Zheng \cite{MiaoMurphyZheng2014} and Dodson-Miao-Murphy-Zheng \cite{Dodson2017} established Conjecture \ref{CNLS0} for general initial data in $d = 4$ with $1 < s_c \leq \frac{3}{2}$. For $d = 4$ and $\frac{3}{2} < s_c < 2$, similar results for radial data were obtained by Lu and Zheng \cite{LuZheng2017}.
	For related results in   dimensions  $d\ge5$, we refer the reader to Zhao \cite{Zhao2017AMS} and Li–Li \cite{LiLi2022SIAM}. See Table \ref{table1} for a summary.

	\begin{table}[h]
		\centering
		\caption{Results for Conjecture \ref{CNLS0} in the super-critical case: $1<s_c<\frac{d}{2}$}\label{table1}
		\begin{tabular}{|c|c|}
			\hline
			$d=3$ & $1<s_c<\frac{3}{2}$, \textcolor{blue}{radial}, Murphy \cite{Murphy2015}\\
			\hline 
			$d=4$ & \thead {  $1<s_c<\frac{3}{2}$, Miao-Murphy-Zheng\cite{MiaoMurphyZheng2014}; $s_c=\frac{3}{2}$, Dodson-Miao-Murphy-Zheng\cite{Dodson2017}; \\  $\frac{3}{2}<s_c<2$, \textcolor{blue}{radial},  Lu-Zheng\cite{LuZheng2017}}\\
			\hline 
			$d\ge5$  & \thead {$1<s_c<\frac{d}{2}$, and \textcolor{blue}{assume  $\alpha $ is even when  $d\ge8$}, \\
				Killip-Visan\cite{KillipVisan2010}, Zhao\cite{Zhao2017AMS}, Li-Li\cite{LiLi2022SIAM}}\\
			\hline
		\end{tabular}
	\end{table}

	\subsection{Main results}
	Analogous to Conjecture \ref{CNLS0}, it is conjectured that for the inhomogeneous NLS (\ref{INLS}):
	\begin{Conjection}\label{CNLS}
		Let $d \geq 1$, $0<b< 2$,  $\alpha \geq \frac{4-2b}{d}$,  and $s_c := \frac{d}{2} - \frac{2-b}{\alpha }$. Assume $u: I \times \mathbb{R}^d \rightarrow \mathbb{C}$ is a maximal-lifespan solution to (\ref{INLS}) such that 
		\begin{equation}
			u \in L_t^\infty \dot{H}_x^{s_c}(I \times \mathbb{R}^d), \notag
		\end{equation}
		then $u$ is global and scatters as $t \to \pm \infty$.
	\end{Conjection}
 
	  Wang and the second author of this paper \cite{WangXu2023JMAA}  first addressed Conjecture \ref{CNLS} in the subcritical case ($d \geq 5, s_c = \frac{1}{2}$) by employing the Lin–Strauss Morawetz inequality. In this work, we establish Conjecture \ref{CNLS} in the energy-supercritical regime in  four-dimensional case. Our main result is stated as follows.
	\begin{theorem}\label{T1}
		Let $1<s_c <2$, $0<b<\min \left\{ (s_c-1)^2+1,3-s_c \right\}$ and   $\alpha >0$ be such that  $s_c=2-\frac{2-b}{\alpha }$.  Suppose $u: I \times \mathbb{R}^4 \to \mathbb{C}$ is a maximal-lifespan solution to \eqref{INLS} such that 
		\begin{equation}
			u \in L_t^\infty \dot{H}_x^{s_c}(I \times \mathbb{R}^4). \label{Ebound}
		\end{equation}
		Then $u$ is global and scatters. 
	\end{theorem}
\begin{remark}
 To avoid technical complications, in Theorem \ref{T1}, we assume the condition $b < (s_c - 1)^2 + 1$, which is equivalent to $s_c < \alpha$. Under this assumption, the proof of the stability theorem becomes fairly straightforward (see Lemma \ref{Lemmanonlinearestimate} below), since the derivative of the nonlinearity possesses enough  regularity when applying the operator $|\nabla|^{s_c}$.  When $b \ge (s_c - 1)^2 + 1$, the stability theory becomes considerably more delicate, as this corresponds to the regime of low-power nonlinearities. In such cases, one needs to refine the function spaces used in the stability framework. We refer the reader to \cite[Section 3.4]{KillipVisan2010} for further discussions and relevant references.
\end{remark}
\begin{remark}
The constraint $b < 3 - s_c$ comes  from the Strichartz estimate (Proposition~\ref{P:SZ}) and  the singularity of the inhomogeneous coefficient $|x|^{-b}$.
 Using the weighted Strichartz estimates as in \cite{Lee2025arXiv}, it may be possible to extend the admissible range of $b$ in Theorem \ref{T1} to include the case $b \ge 3 - s_c$.
	\end{remark}
	
	Similar results for the classical nonlinear Schr\"{o}dinger equation in dimension four have been established previously by Miao-Murphy-Zheng \cite{MiaoMurphyZheng2014}, Dodson-Miao-Murphy-Zheng \cite{Dodson2017} and Lu-Zheng \cite{LuZheng2017}. Notably, in the range $\frac{3}{2}<s_c<2$, their results require the initial data to satisfy a radial symmetry assumption (c.f. Table \ref{table1}). In this paper, we show that this radial symmetry condition can be removed for the INLS equation. The key reason for this improvement is the presence of the spatially decaying factor $|x|^{-b}$ in the nonlinearity. This decay at infinity prevents the minimal almost periodic solution from concentrating at infinity in either physical or frequency space (see Proposition \ref{Pscatteringsolution} below).  Consequently, both spatial and frequency translation parameters must remain bounded. Thus, after applying a suitable translation, these parameters can be taken to be zero. This effectively places us in the same scenario as in the radial case but without the need for a radial assumption.  
	
	\subsection{Outline of the proof of Theorem \ref{T1}}
	We need some  definitions:
	\begin{definition}[Strong solution]
		Let
		\begin{equation}
			\gamma = 2(\alpha +1), \qquad m =  \frac{4\alpha (\alpha +1)}{\alpha +2-b(\alpha +1)}.\notag
		\end{equation}
		A function $u: I \times \mathbb{R}^4 \to \mathbb{C}$ is a solution to \eqref{INLS} if, for any compact $J \subset I$, $u \in C_t^0 L_x^2(J \times \mathbb{R}^4) \cap L_t^\gamma L_x^{m,2}(J \times \mathbb{R}^4)$, and satisfies the Duhamel formula for all $t, t_0 \in I$:
		\begin{equation}
			u(t) = e^{i(t - t_0)\Delta} u(t_0) - i \int_{t_0}^t e^{i(t-\tau)\Delta} |x|^{-b} |u(\tau)|^\alpha  u(\tau) \, d\tau, \notag
		\end{equation}
		where $L_x^{m,2}$ is the Lorentz space (defined in Subsection \ref{S:2.2}). We refer to $I$ as the \textit{lifespan} of $u$. If $u$ cannot be extended to any strictly larger interval, it is called a \textit{maximal-lifespan solution}. If $I = \mathbb{R}$, $u$ is called a \textit{global solution}.
	\end{definition}

	It is natural to assume that solutions belong locally in time to the space $L_t^\gamma L_x^{m,2}(I\times \mathbb{R}^4)$, as the linear flow always lies in this space by the Strichartz estimate (see proposition \ref{P:SZ} below). 
	Furthermore, any global solution $u$ of (\ref{INLS}) satisfying $\|u\|_{L_t^\gamma L_x^{m,2}(\mathbb{R}\times \mathbb{R}^4)}<+\infty$ necessarily scatters in both time directions 
	in the sense that there exist unique scattering states  $u_\pm\in \dot H^{s_c}(\mathbb{R} ^4)$ such that 
	\begin{equation}
		\|u(t)-e ^{it\Delta }u_\pm\|_{\dot H_x^{s_c}(\mathbb{R} ^4)}\rightarrow0\quad\text{as}\quad t\rightarrow\pm \infty. \notag
	\end{equation}
	In view of this, we introduce the scattering size of the solution $u$ as
	\begin{equation}
		S_I(u)= \|u\|_{L_t^\gamma L_x^{m,2}(I\times \mathbb{R}^4)}.\notag
	\end{equation}
	Closely associated with the notion of scattering is the notion of blow-up:
	\begin{definition}[Blow-up]\label{Blowup}
		A solution $u$ to \eqref{INLS} is said to blow up forward in time if there exists $t_0 \in I$ such that
		\begin{equation}
			\|u\|_{ S_{[t_0, \sup(I)) }} = \infty.\notag
		\end{equation}
		Similarly, $u$ blows up backward in time if there exists $t_0 \in I$ such that
		\begin{equation}
			\|u\|_{ S_{(\inf(I), t_0]  }} = \infty.\notag
		\end{equation}
	\end{definition}
	 
	 The local theory for (\ref{INLS}) has been worked out in \cite{AnKim2023ZAA}. 
	\begin{theorem}[Local well-posedness]\label{TLocalwellposedness}
		Let $1 < s_c < 2,0<b<2$ and  $\alpha >0$ be such that  $s_c=2-\frac{2-b}{\alpha }$.  Given $u_0 \in \dot{H}^{s_c}(\mathbb{R}^4)$ and $t_0 \in \mathbb{R}$, there exists a unique maximal-lifespan solution $u : I \times \mathbb{R}^4 \to \mathbb{C}$ to (\ref{INLS}) with $u(t_0) = u_0$. Moreover, this solution satisfies the following:
		\begin{enumerate}
			\item \textbf{(Local existence)} $I$ is an open neighborhood of $t_0$.
			\item \textbf{(Blowup criterion)} If $\sup I$ is finite, then $u$ blows up forward in time (in the sense of Definition \ref{Blowup}). Similarly, if $\inf I$ is finite, then $u$ blows up backward in time.
			\item \textbf{(Scattering)} If $\sup I = +\infty$ and $u$ does not blow up forward in time, then $u$ scatters forward in time. Specifically, there exists a unique $u_+ \in \dot{H}^{s_c}_x(\mathbb{R}^4)$ such that
			\[
			\lim_{t \to \infty} \|u(t) - e^{it\Delta}u_+\|_{\dot{H}^{s_c}_x(\mathbb{R}^4)} = 0. \label{1.4}
			\]
			Conversely, given $u_+ \in \dot{H}^{s_c}_x(\mathbb{R}^4)$, there is a unique solution to (\ref{INLS}) in a neighborhood of $t = \infty$ such that \eqref{1.4} holds. Analogous statements hold backward in time.
			\item \textbf{(Small-data global existence and scattering)} If $\|u_0\|_{\dot{H}^{s_c}_x}$ is sufficiently small, then $u$ is global and scatters.
		\end{enumerate}
	\end{theorem}
 
	  To prove  Theorem \ref{T1}, following the road map in \cite{KenigMerle2006}, we argue by contradiction. First, define $L: [0,\infty) \to [0,\infty]$ by
	\[
	L(E) := \sup\{S_I(u) \mid u: I \times \mathbb{R} \to \mathbb{C} \text{ solving } (\ref{INLS}) \text{ with } \|u\|^2_{L_t^\infty \dot{H}_x^{s_c}(I \times \mathbb{R}^4)} \leq E \}.
	\]
From the stability theory (see Theorem \ref{TStability} below), we know that  $L$ is continuous. 	Note also that $L(E)$ is a nonnegative, nondecreasing   function, and by Theorem \ref{TLocalwellposedness} $L(E) \leq E^\alpha $ for $E$ sufficiently small. Then there must exist a unique critical threshold $E_c \in (0,\infty]$ such that
	\[
	L(E) < \infty \quad \text{if } E < E_c, \qquad
	L(E) = \infty \quad \text{if } E \ge E_c.
	\]
	The failure of main theorem implies that $0 < E_c < \infty$. Moreover, 
	  \begin{theorem}[existence of minimal counterexamples]\label{TReduction}
		Suppose  Theorem \ref{T1} fails to be true. 
		Then, there exists a maximal life-span solution   $u:[0,T_{\text{max}})\times \mathbb{R} ^4\rightarrow \mathbb{C}$ such that the following hold:
\begin{itemize}
	\item[(1)] $u$ blows up its scattering norm forward in time, i.e. $\mathrm{S}_{[0,T_{\max})}(u) = +\infty$;
	
	\item[(2)] there exists a frequency scale function $N : [0, T_{\max}) \to (0, \infty)$ such that
\begin{equation}
		\left\{ N(t)^{s_c-2} u(t, N(t)^{-1} x) : t \in [0, T_{\max}) \right\}\label{E592}
\end{equation}
	is precompact in $\dot{H}^{s_c}(\mathbb{R}^4)$.
\end{itemize}
In addition, we may assume   $\displaystyle \inf _{t\in [0,T_{\text{max}})}N(t)\ge1$. 
	\end{theorem}
	\begin{remark}
		By Arzala-Ascoli theorem,  (\ref{E592}) implies that for any  $\eta>0, $ there exists  $C(\eta) >0$ such that 
		\begin{equation}
		\int _{|x|\ge \frac{C(\eta)}{N(t)}}||\nabla|^{s_c}u(t,x)|^2dx+	\int _{|\xi|\ge C(\eta)N(t)}|\xi|^{2s_c}|\hat u(t.\xi)|^2d\xi \le \eta \label{Ecompact}
		\end{equation}
		\end{remark}
	In the rest of this paper, we will call the solutions found in Theorem \ref{TReduction} as \emph{almost periodic solutions}. 
The detailed construction of such solutions in the context of the standard $\dot{H}^{s_c}$-critical NLS was carried out in \cite{KillipVisan2010, Murphy2014}. In the $\dot{H}^1$-critical setting of the (INLS), a similar construction was established by Guzman and Murphy \cite{GuzmanMurphy2021JDE}. Following the arguments developed in \cite{KillipVisan2010, Murphy2014, GuzmanMurphy2021JDE}, the key ingredients in the construction are the linear profile decomposition, as developed by Shao \cite{Shao2009EJDE}, and the embedding of nonlinear profiles (will be presented in detail in  Proposition \ref{Pscatteringsolution}).

  	To complete the proof of Theorem \ref{T1}, it sufficies to  rule out the existence of almost periodic solutions as described in Theorem \ref{TReduction}.   We first note that for the maximal-lifespan solution found in  Theorem  \ref{TReduction},  $N(t)$  enjoys the following properties: Lemma \ref{LLocal constant}, Corollary \ref{C1.9} and Lemma \ref{LSpacetime bounds} (see \cite{KillipVisan2013} for details). 
 
	\begin{lemma}[Local constancy]\label{LLocal constant}
		Let $u: I \times \mathbb{R}^4 \to \mathbb{C}$ be a maximal-lifespan almost periodic solution to \eqref{INLS}. Then there exists $\delta = \delta(u) > 0$ such that for all $t_0 \in I$,
		\[
		[t_0 - \delta N(t_0)^{-2}, t_0 + \delta N(t_0)^{-2}] \subset I.
		\]
		Moreover, $N(t) \sim_u N(t_0)$ for $|t - t_0| \leq \delta N(t_0)^{-2}$.
	\end{lemma}

	\begin{corollary}[$N(t)$ at blow-up]\label{C1.9}
		Let $u : I \times \mathbb{R}^4 \to \mathbb{C}$ be a \textit{maximal-lifespan almost periodic solution} to \eqref{INLS}. If $T$ is a finite endpoint of $I$, then $N(t) \gtrsim_u |T - t|^{-1/2}$. In particular, $\lim_{t \to T} N(t) = \infty$. If $I$ is infinite or semi-infinite, then for any $t_0 \in I$, we have $N(t) \gtrsim_u \langle t - t_0 \rangle^{-1/2}$.
	\end{corollary}

	\begin{lemma}[Spacetime bounds]\label{LSpacetime bounds}
		Let $u: I \times \mathbb{R}^4 \to \mathbb{C}$ be an almost periodic solution to \eqref{INLS}. Then, the following bound holds:
		\[
		\int_I N(t)^2 \, dt \lesssim_u \| |\nabla|^{s_c} u \|_{L_t^q L_x^{r,2}(I \times \mathbb{R}^4)}^q \lesssim_u 1 + \int_I N(t)^2 \, dt,
		\]
		where $(q, r)$ is an admissible pair (see Definition 2.2 below) with $q < \infty$.
	\end{lemma}
	
	In  Section \ref{S4}, using  Lemma \ref{LLocal constant}, Corollary \ref{C1.9}, Lemma \ref{LSpacetime bounds} and the space-localized Morawetz inequality  employed in Bourgain \cite{Bourgain1999} in the study of the radial energy-critical NLS,  we rule out the  almost periodic solutions  in  the case $1 < s_c < 3/2$.

	In the case  $\frac{3}{2}\le s_c<2$, we first refine the class of almost periodic solutions. Using rescaling arguments as in \cite{KillipTaoVisan2009,KillipVisan2010AJM,TaoVisanZhang2007}, 
	we can ensure that the almost periodic solutions under consideration do not escape to arbitrarily low frequencies on at least half of their maximal lifespan, say $[0, T_{\text{max}})$.  
	By applying Lemma 1.5 to partition $[0, T_{\text{max}})$ into characteristic subintervals $J_k$, we obtain the following theorem. 
	\begin{theorem}[Two scenarios for blow-up]\label{TReduction2}
		If Theorem \ref{T1} does not hold, then there exists an almost periodic solution $u: [0, T_{\text{max}}) \times \mathbb{R}^4 \to \mathbb{C}$ that blows up forward in time and satisfies
		\[
		u \in L_t^\infty \dot{H}_x^{s_c}([0, T_{\text{max}}) \times \mathbb{R}^4).
		\]
		Additionally, $[0, T_{\text{max}})$ can be decomposed as $[0, T_{\text{max}}) = \bigcup_k J_k$, where
		\[
		N(t) \equiv N_k \geq 1 \quad \text{for} \quad t \in J_k, \quad \text{with} \quad |J_k| \sim_u N_k^{-2}.
		\]
		The solution $u$ falls into one of the following two cases:
		\[
		\int_0^{T_{\text{max}}} N(t)^{3-2s_c} \, dt < \infty \quad \text{(rapid frequency cascade solution)},
		\]
		or
		\[
		\int_0^{T_{\text{max}}} N(t)^{3-2s_c} \, dt = \infty \quad \text{(quasi-soliton solution)}.
		\]
	\end{theorem}
	In view of this theorem, to prove Theorem \ref{T1} in the case  $\frac{3}{2}\le s_c<2$, it suffices to rule out all the possible scenarios described in Theorem \ref{TReduction2}.
	
	In Section \ref{S5}, we exclude the possibility of "rapid frequency-cascade solutions." Two key analytical tools are employed to achieve this. The first is the long-time Strichartz estimate, originally introduced by Dodson \cite{Dodson2012} in the context of mass-critical nonlinear Schrödinger equations. For further applications of this method to the energy-critical, energy-subcritical, and energy-supercritical regimes, see \cite{Visan2012,Murphy2014,Murphy2015}. See also \cite{LiuMiaoZheng}, where a long-time Strichartz estimate was established specifically for the mass-critical inhomogeneous NLS, which is more closely related to the supercritical INLS considered in this paper. 
	The second tool is the following "No-waste Duhamel formula" (see Proposition 5.23 in \cite{KillipVisan2013}):
	\begin{lemma}[No-waste Duhamel formula]\label{LNoWasteDuhamel}
		If $u$ is an almost periodic solution on the maximal lifespan $I$, then for any $t \in I$, we have
		\begin{equation}
			u(t) = \lim_{T \to T_{\text{max}}} i \int_t^T e^{i(t-s)\Delta}|x|^{-b} |u(s,x)|^{\alpha} u(s,x) \, ds,\notag
		\end{equation}
		where the equality holds in the weak limit sense in $\dot H^{s_c}(\mathbb{R}^4)$.
	\end{lemma}
	
	Using the long-time Strichartz estimate and the "No-waste Duhamel formula," we deduce that "rapid frequency-cascade solutions" must have zero mass, which contradicts the assumption that the solution blows up. Hence, "rapid frequency-cascade solutions" can be ruled out.
	
	In Section \ref{S6}, we exclude the existence of "quasi-soliton solutions". The primary tool employed here is the frequency-localized Morawetz inequality (Theorem \ref{TMorawetz}). The key idea in Theorem \ref{TMorawetz} is to apply a high-frequency truncation to the standard Morawetz inequality and then estimate the error terms. For these error estimates, the long-time Strichartz estimate established earlier plays a crucial role. The frequency-localized Morawetz inequality provides a uniform bound on $\int_I N(t)^{3-2s_c} \, dt$ for any compact interval $I \subset [0, T_{\text{max}})$. By choosing $I \subset [0, T_{\text{max}})$ sufficiently large, we arrive at a contradiction, thereby excluding the existence of "quasi-soliton solutions." This completes the proof of Theorem \ref{T1}.
 
	\section{Preliminaries}\label{S:2}
	\subsection{Some notations}
	We use the standard notation for mixed Lebesgue space-time norms and Sobolev spaces. 
	We write  $A\lesssim  B$ to denote  $A\le CB$ for some  $C>0$. If  $A\lesssim B$ and  $B\lesssim A$, then we write  $A\approx B$.  We write  $A\ll B$   to denote  $A\le cB$ for some  small $c>0$.  If  $C$ depends upon some additional parameters, we will indicate this with subscripts; for example,  $X\lesssim _u Y$ denotes that  $X\le C_u Y$ for some  $C_u$ depending on  $u$.    We use  $O(Y)$ to denote any quantity  $X$  such that  $|X|\lesssim  Y$. We use     $\langle x \rangle $ to denote  $(1+|x|^2)^{\frac{1}{2}}$.        We write $L_t^q L_x^r$ to denote the Banach space with norm
	\[
	\|u\|_{L_t^q L_x^r (\mathbb{R} \times \mathbb{R}^d)} := \left( \int_{\mathbb{R}} \left( \int_{\mathbb{R}^d} |u(t, x)|^r \, dx \right)^{q/r} \, dt \right)^{1/q},
	\]
	with the usual modifications when $q$ or $r$ are equal to infinity, or when the domain $\mathbb{R} \times \mathbb{R}^d$ is replaced by spacetime slab such as $I \times \mathbb{R}^d$. When $q = r$ we abbreviate $L_t^q L_x^q$ as $L_{t,x}^q$. 
	
	Let $\phi(\xi)$ be a radial bump function supported in the ball $\{\xi \in \mathbb{R}^d : |\xi| \leq 2\}$ and equal to  $1$ on the ball $\{\xi \in \mathbb{R}^d : |\xi| \leq 1\}$. For each number $N > 0$, we define the Fourier multipliers
	\begin{equation}
		\widehat{P_{\leq N} f}(\xi) := \phi(\xi/N) \widehat{f}(\xi), \quad \widehat{P_{> N} f}(\xi) := (1 - \phi(\xi/N)) \widehat{f}(\xi),\notag
	\end{equation}
	\begin{equation}
		\widehat{P_N f}(\xi) = :\psi(\xi/N) \widehat{f}(\xi) := (\phi(\xi/N) - \phi(2\xi/N)) \widehat{f}(\xi).\notag
	\end{equation}
	We similarly define $P_{< N}$ and $P_{\geq N}$.  For convenience of notation, let $u_N := P_{N}u$, $u_{\leq N} := P_{\leq N}u$, and $u_{>N} := P_{>N}u$.
	We will normally use these multipliers when $M$ and $N$ are dyadic numbers (that is, of the form $2^n$ for some integer $n$); in particular, all summations over $N$ or $M$ are understood to be over dyadic numbers.  
	
	\subsection{Lorentz spaces and  useful lemmas}\label{S:2.2}
	Let $f$ be a measurable function on $\mathbb{R}^d$. The distribution function of $f$ is defined by
	\begin{equation}
		d_f(\lambda):= |\{x\in \mathbb{R}^d : |f(x)|>\lambda\}|, \quad \lambda>0,\notag
	\end{equation}
	where $|A|$ is the Lebesgue measure of a set $A$ in $\mathbb{R}^d$. The decreasing rearrangement of $f$ is defined by
	\begin{equation}
		f^*(s):= \inf \left\{ \lambda>0 : d_f(\lambda)\leq s\right\}, \quad s>0.\notag
	\end{equation}
	
	\begin{definition}[Lorentz spaces] ~\\
		Let $0<r<\infty$ and $0<\rho\leq \infty$. The Lorentz space $L^{r,\rho}(\mathbb{R}^d)$ is defined by
		\begin{equation}
			L^{r,\rho}(\mathbb{R}^d):= \left\{ f \text{ is measurable on } \mathbb{R}^d : \|f\|_{L^{r,\rho}}<\infty\right\}, \notag
		\end{equation}
		where
		\[
		\|f\|_{L^{r,\rho}}:= \left\{
		\begin{array}{cl}
			( \frac{\rho}{r} \int_0^\infty (s^{1/r} f^*(s))^\rho \frac{1}{s}ds)^{1/\rho} &\text{ if } \rho <\infty, \\
			\sup_{s>0} s^{1/r} f^*(s) &\text{ if } \rho=\infty.
		\end{array}
		\right.
		\]
	\end{definition}
	
	We collect the following basic properties of $L^{r,\rho}(\mathbb{R}^d)$ in the following lemmas.
	\begin{lemma}[Properties of Lorentz spaces \cite{ONeil}] \ 
		\begin{itemize}
			\item For $1<r<\infty$, $L^{r,r}(\mathbb{R}^d) \equiv L^r(\mathbb{R}^d)$ and by convention, $L^{\infty,\infty}(\mathbb{R}^d)= L^\infty(\mathbb{R}^d)$.
			\item For $1<r<\infty$ and $0<\rho_1<\rho_2\leq \infty$, $L^{r,\rho_1}(\mathbb{R}^d)\subset L^{r,\rho_2}(\mathbb{R}^d)$. 
			\item For $1<r<\infty$, $0<\rho \leq \infty$, and $\theta>0$, $\||f|^\theta\|_{L^{r,\rho}} = \|f\|^\theta_{L^{\theta r, \theta \rho}}$.
			\item For $b>0$, $|x|^{-b} \in L^{\frac{d}{b},\infty}(\mathbb{R}^d)$ and $\||x|^{-b}\|_{L^{\frac{d}{b},\infty}} = |B(0,1)|^{\frac{b}{d}}$, where $B(0,1)$ is the unit ball of $\mathbb{R}^d$.
		\end{itemize}
	\end{lemma}
	\begin{lemma}[H\"older's inequality \cite{ONeil}]  \ 
		\begin{itemize}
			\item  Let $1<r, r_1, r_2<\infty$ and $1\leq \rho, \rho_1, \rho_2 \leq \infty$ be such that
			\[
			\frac{1}{r}=\frac{1}{r_1}+\frac{1}{r_2}, \quad \frac{1}{\rho} \leq \frac{1}{\rho_1}+\frac{1}{\rho_2}.
			\]
			Then for any $f \in L^{r_1, \rho_1}(\mathbb{R}^d)$ and $g\in L^{r_2, \rho_2}(\mathbb{R}^d)$
			\[
			\|fg\|_{L^{r,\rho}} \lesssim \|f\|_{L^{r_1, \rho_1}} \|g\|_{L^{r_2,\rho_2}}.
			\]
			\item Let  $1<r_1,r_2<\infty $ and  $1\le \rho_1,\rho_2\le \infty $  be such that 
			\begin{equation}
				1=\frac{1}{r_1}+\frac{1}{r_2},\quad 1\le \frac{1}{\rho_1}+\frac{1}{\rho_2}.\notag
			\end{equation}
			Then  for any $f \in L^{r_1, \rho_1}(\mathbb{R}^d)$ and $g\in L^{r_2, \rho_2}(\mathbb{R}^d)$
			\begin{equation}
				\|fg\|_{L^1}\lesssim  \|f\|_{L^{r_1, \rho_1}} \|g\|_{L^{r_2,\rho_2}}.\notag
			\end{equation}
		\end{itemize}
	\end{lemma}

	\begin{lemma}[Convolution inequality \cite{ONeil}] ~
		Let $1<r,r_1, r_2<\infty$ and $1\leq \rho, \rho_1, \rho_2 \leq \infty$ be such that
		\[
		1+\frac{1}{r} =\frac{1}{r_1}+\frac{1}{r_2}, \quad \frac{1}{\rho}\leq \frac{1}{\rho_1}+\frac{1}{\rho_2}.
		\]
		Then 	for any $f\in L^{r_1, \rho_1}(\mathbb{R}^d)$ and $g\in L^{r_2, \rho_2}(\mathbb{R}^d)$ 
		\[
		\|f\ast g\|_{L^{r,\rho}(\mathbb{R} ^d)} \lesssim \|f\|_{L^{r_1,\rho_1}(\mathbb{R} ^d)} \|g\|_{L^{r_2,\rho_2}(\mathbb{R} ^d)}. 
		\]
	\end{lemma}

	Using the above convolution inequality,  one can easily obtain    Bernstein's inequality in Lorentz spaces.
	\begin{lemma}[Bernstein's inequality]\label{L:Bernstein}
		Let  $N>0$,  $1<r_1<r_2<\infty $ and  $1\le \rho_1\le\rho_2\le \infty $. Then 
		\begin{equation}
			\|P_{N}f\|_{L^{r_2,\rho_2}(\mathbb{R} ^d)}\lesssim  N^{d(\frac{1}{r_1}-\frac{1}{r_2})} \|f\|_{L^{r_1,\rho_1}(\mathbb{R} ^d)}.\notag  
		\end{equation}
	\end{lemma}
	Next,  in Lemma \ref{L:sobolev}--Lemma \ref{L:6141},  we  recall the Sobolev embedding, Hardy's inequality, product rule, and chain rule in Lorentz spaces. We start by introducing the following definition.
	\begin{definition}
		Let  $s\ge0,1<r<\infty $ and  $1\le\rho\le \infty $. We define the Sobolev-Lorentz spaces 
		\begin{equation}
			W^sL^{r,\rho}(\mathbb{R} ^d):=\left\{f\in \mathcal{S}'(\mathbb{R} ^d):(1-\Delta )^{s/2}f\in L^{r,\rho}(\mathbb{R} ^d) \right\},\notag
		\end{equation}  
		\begin{equation}
			\dot W^sL^{r,\rho}(\mathbb{R} ^d):=\left\{f\in \mathcal{S}'(\mathbb{R} ^d):(-\Delta )^{s/2}f\in L^{r,\rho}(\mathbb{R} ^d) \right\},\notag
		\end{equation}
		where  $\mathcal{S}'(\mathbb{R} ^d)$ is the space of tempered distributions on  $\mathbb{R} ^d$ and 
		\begin{equation}
			(1-\Delta )^{s/2}f=\mathcal{F}^{-1}\left((1+|\xi|^2)^{s/2}\mathcal{F} (f)\right),\qquad (-\Delta )^{s/2}f=\mathcal{F} ^{-1}(|\xi|^s\mathcal{F} (f))\notag
		\end{equation}  
		with  $\mathcal{F} $ and  $\mathcal{F} ^{-1}$ the Fourier and its inverse Fourier transforms respectively. The spaces  $W^sL^{r,\rho}(\mathbb{R} ^d)$ and  $\dot W^sL^{r,\rho}(\mathbb{R} ^d)$ are endowed respectively with the norms
		\begin{equation}
			\|f\|_{W^sL^{r,\rho}} = \|f\|_{L^{r,\rho}}+ \|(-\Delta )^{s/2}f\|_{L^{r,\rho}},\qquad  \|f\|_{\dot WL^{r,\rho}} = \|(-\Delta )^{s/2}f\|_{L^{r,\rho}}.\notag   
		\end{equation}    
		For simplicity, when  $s=1$ we write  $WL^{r,\rho}:=WL^{r,\rho}(\mathbb{R} ^d)$  and  $\dot WL^{r,\rho}:=\dot WL^{r,\rho}(\mathbb{R} ^d)$.  
	\end{definition}
	
	\begin{lemma}[Sobolev embedding \cite{DinhKe}]\label{L:sobolev}
		Let  $1<r<\infty ,1\le \rho\le \infty $ and  $0<s<\frac{d}{r}$. Then 
		\begin{equation}
			\|f\|_{L^{\frac{dr}{d-sr},\rho}(\mathbb{R} ^d)}\lesssim   \|(-\Delta )^{s/2}f\|_{L^{r,\rho}(\mathbb{R} ^d)} \quad\text{for any}\quad   f\in \dot W^sL^{r,\rho}(\mathbb{R} ^d).\notag
		\end{equation}  
	\end{lemma}

	\begin{lemma}[Hardy's inequality \cite{DinhKe}]\label{LHardy}
		 Let $1 < r < \infty$, $1 \leq \rho \leq \infty$ and $0 < s < \frac{d}{r}$. Then  
		\[
		\||x|^{-s} f\|_{L^{r,\rho}(\mathbb{R} ^d)} \lesssim  \|(-\Delta)^{s/2} f\|_{L^{r,\rho}(\mathbb{R} ^d)} \quad\text{for any}\quad   f\in \dot W^sL^{r,\rho}(\mathbb{R} ^d).
		\]
	\end{lemma}
	\begin{lemma}[Product rule I \cite{Cruz}]\label{L:leibnitz}
		Let  $s\ge0,1<r,r_1,r_2,r_3,r_4,<\infty $, and   $1\le\rho,\rho_1,\rho_2,\rho_3,\rho_4\le\infty $ be such that 
		\begin{equation}
			\frac{1}{r}=\frac{1}{r_1}+\frac{1}{r_2}=\frac{1}{r_3}+\frac{1}{r_4},\qquad \frac{1}{\rho}=\frac{1}{\rho_1}+\frac{1}{\rho_2}=\frac{1}{\rho_3}+\frac{1}{\rho_4}.\notag
		\end{equation} 
		Then for any  $f\in \dot W^sL^{r_1,\rho_1}(\mathbb{R} ^d)\cap L^{r_3,\rho_3}(\mathbb{R} ^d)$  and  $g\in \dot W^sL^{r_4,\rho_4}(\mathbb{R} ^d)\cap L^{r_2,\rho_2}(\mathbb{R} ^d)$,  
		\begin{equation}
			\|(-\Delta )^{s/2}(fg)\|_{L^{r,\rho}} \lesssim   \|(-\Delta )^{s/2}f\|_{L^{r_1,\rho_1}} \|g\|_{L^{r_2,\rho_2}}+ \|f\|_{L^{r_3,\rho_3}} \|(-\Delta )^{s/2}g\|_{L^{r_4,\rho_4}}.\notag    
		\end{equation}
	\end{lemma}

	\begin{lemma}[Chain rule \cite{Aloui}]\label{L:6141}
		Let  $s\in [0,1],F\in C^1(\mathbb{C},\mathbb{C})$ and  $1<p,p_1,p_2<\infty $,  $1\le q,q_1,q_2<\infty $ be  such that 
		\begin{equation}
			\frac{1}{p}=\frac{1}{p_1}+\frac{1}{p_2},\qquad\frac{1}{q}=\frac{1}{q_1}+\frac{1}{q_2}.\notag
		\end{equation}   
		Then
		\begin{equation}
			\|(-\Delta )^{s/2}F(f)\|_{L^{p,q}(\mathbb{R} ^d)}\lesssim  \|F'(f)\|_{L^{p_1,q_1}(\mathbb{R} ^d)} \|(-\Delta )^{s/2}f\|_{L^{p_2,q_2}(\mathbb{R} ^d)}.\notag   
		\end{equation}
	\end{lemma}
	
	Finally, we record several nonlinear estimates in Sobolev spaces that will be used in subsequent analysis.
	
		\begin{lemma}[Product rule II \cite{CCF}]\label{L-qd}
		Let $b>0$, $1<p<\frac{4}{b},$ and $0\leq s<\frac{4}{p}-b$. For sufficiently small $\epsilon>0$, we have 
		$$ \||\nabla|^s(|x|^{-b}f)\|_{L^p(\mathbb{R}^4)}\lesssim_\epsilon \Big[\||\nabla|^sf\|_{L^{q^+}(\mathbb{R}^4)}\||\nabla|^sf\|_{L^{q^-}(\mathbb{R}^4)}\Big]^{\frac12},$$
		where $\frac{1}{q^\pm}:=\frac{1}{p}- \frac{b\pm\epsilon}{4}$.
	\end{lemma}
	
	\begin{lemma}[Paraproduct estimate  \cite{MiaoMurphyZheng2014}]\label{Lemmaparaproduct}
		Let $0 < s < 1$. If $1 < r < r_1 < \infty$ and $1 < r_2 < \infty$ satisfy
		\[
		\frac{1}{r_1} + \frac{1}{r_2} = \frac{1}{r} + \frac{s}{4} < 1,
		\]
		then
		\[
		\| |\nabla|^{-s}(fg) \|_{L^r(\mathbb{R}^4)} \lesssim \| |\nabla|^{-s} f \|_{L^{r_1}(\mathbb{R}^4)} \| |\nabla|^{s} g \|_{L^{r_2}(\mathbb{R}^4)}.
		\]
	\end{lemma}
	\begin{lemma}\label{LVisan}
			Let $F$ be a Hölder continuous function of order $0 < \alpha < 1$. Then, for every $0 < \sigma < \alpha$, $1 < p < \infty$, and $\sigma/\alpha < s < 1$, we have
		\[
		\| |\nabla|^\sigma F(u)\|_{L^p} \lesssim \||u|^{\alpha - \sigma/s}\|_{L^{p_1}} \| |\nabla|^s u\|_{L^{(\sigma/s)p_2}}^{\sigma/s},
		\]
		provided $1/p = 1/p_1 + 1/p_2$ and $(1 - \sigma/(\alpha s))p_1 > 1$.
	\end{lemma}
	
	\subsection{Strichartz estimate  and stability.}
	In this subsection, we recall Strichartz estimate in the Lorentz space  and the stability results of the Cauchy probelm (\ref{INLS}).  
	\begin{definition}[Admissibility] 
		A pair $(q,r)$ is said to be Schr\"odinger admissible, for short $(q,r)\in \Lambda $, where
		\begin{equation}
			\Lambda=\left\{(q,r): 2\le q,r\le \infty,\  	\frac{2}{q}+\frac{4}{r}=2\right\}.\notag
		\end{equation}
	\end{definition}
	
	The following result is a key tool for our work. It is
	established in \cite[Theorem 10.1]{KeelTao1998AJM}.
	\begin{proposition}[Strichartz estimates]\label{P:SZ}\ \\
		\begin{itemize}
			\item Let $(q,r)\in \Lambda $ with $r<\infty$. Then for any $f\in L^2(\mathbb{R}^4)$
			\begin{align}  
				\|e^{it\Delta }f\|_{L^q_t L^{r,2}_x(\mathbb{R}\times \mathbb{R}^4)} \lesssim \|f\|_{L^2_x(\mathbb{R}^4)}.\notag
			\end{align}
			
			\item Let $(q_1, r_1), (q_2,r_2)\in \Lambda $ with $r_1, r_2<\infty$, $t_0\in \mathbb{R}$ and $I\subset \mathbb{R}$ be an interval containing $t_0$. Then 		for any $F\in L_t^{q_2'}L^{r_2',2}_x(I\times \mathbb{R}^4)$ 
			\begin{align} 
				\left\|\int_{t_0}^t e^{i(t-\tau)\Delta } F(\tau) d\tau\right\|_{L^{q_1}_tL^{r_1,2}_x(I\times \mathbb{R}^4)} \lesssim \|F\|_{L^{q_2'}_tL^{r_2',2}_x(I\times \mathbb{R}^4)}.\notag
			\end{align}
		\end{itemize}
	\end{proposition}
	Using Strichartz estimates and standard continuous argument, we can get the classical stability theory for equation (\ref{INLS}) (see \cite{Colliander2008} for more details). 
		\begin{theorem}[Stability]\label{TStability}
		Let $1<s_c <2$, $0<b<\min \left\{ (s_c-1)^2+1,3-s_c \right\}$ and  $\alpha >0$ be such that  $s_c=2-\frac{2-b}{\alpha }$.   Suppose  $I$ is a compact time interval and $\tilde{u} : I \times \mathbb{R}^4 \to \mathbb{C}$ is an approximate solution to \eqref{INLS} in the sense that
		\[
		(i \partial_t + \Delta) \tilde{u} = F(\tilde{u}) + e
		\]
		for some function $e$. Assume that
		\[
		\|\tilde{u}\|_{L_t^\infty \dot{H}_x^{s_c}(I \times \mathbb{R}^4)} \leq E, \quad  \| \widetilde{u}\|_{L_t^{2(\alpha +1)}L_x^{\frac{4\alpha (\alpha +1)}{\alpha +2-b(\alpha +1)},2}(I\times \mathbb{R} ^4)}\leq L
		\]
		for some $E > 0$ and $L > 0$.
		Let $t_0 \in I$ and $u_0 \in \dot{H}^{s_c}(\mathbb{R}^4)$. Then there exists $\varepsilon_0 = \varepsilon_0(E, L)$ such that if
		\[
		\|u_0 - \tilde{u}(t_0)\|_{\dot{H}^{s_c}(\mathbb{R}^4)} \leq \varepsilon, \quad \||\nabla|^{s_c} e\|_{L_t^{2}L_x^{\frac{4}{3},2}(I\times \mathbb{R} ^4)} \leq \varepsilon
		\]
		for some $0 < \varepsilon < \varepsilon_0$, then there exists a solution $u : I \times \mathbb{R}^4 \to \mathbb{C}$ to \eqref{INLS} with $u(t_0) = u_0$ satisfying
		\[
		\||\nabla |^{s_c}(u-\widetilde{u})\|_{L_t^{2(\alpha +1)} L_x^{\frac{4(\alpha +1)}{2\alpha +1},2}(I\times \mathbb{R} ^4)} \lesssim_{E, L} \varepsilon.
		\]
	\end{theorem}

Following the standard arguments presented in \cite{Colliander2008}, to establish the stability theorem \ref{TStability}, it suffices to prove a short-time perturbation result. The proof of this short-time perturbation result, in turn, relies on the following nonlinear estimate:
	\begin{lemma}\label{Lemmanonlinearestimate}
Let  $1<s_c<2,0<b<\min \left\{ (s_c-1)^2+1,3-s_c \right\}$	and $\alpha >0$ be such that  $s_c=2-\frac{2-b}{\alpha }$. 	Let $F(u) = |u|^\alpha u$ and  $(\gamma ,\rho)=(2(\alpha +1),\frac{4(\alpha +1)}{2\alpha +1})$. Then for any $u,v \in L_t^{\gamma} \dot W^{s_c}L_x^{\rho,2}(I \times \mathbb{R}^d)$, the following estimate holds:
		\begin{align}
			& \||\nabla|^{s_c}[|x|^{-b}(F(u+v)-F(u))]\|_{L_t^2L_x^{\frac{4}{3},2}(I \times \mathbb{R}^4)} 
			\lesssim M_{s_c}(u,v), \notag
		\end{align}
		where we have set
		\[
		M_{s_c}(u,v):=\left( \||\nabla|^{s_c}u\|_{L_t^\gamma L_x^{\rho,2}(I\times\mathbb{R}^4)}^\alpha + \||\nabla|^{s_c}v\|_{L_t^\gamma L_x^{\rho,2}(I\times\mathbb{R}^4)}^\alpha \right)\||\nabla|^{s_c}v\|_{L_t^\gamma L_x^{\rho,2}(I\times\mathbb{R}^4)}.
		\]
	\end{lemma}
\begin{remark}
	The assumption $b < (s_c - 1)^2 + 1$ is equivalent to $s_c < \alpha$. This condition ensures that $F'(u)$ possesses enough  regularity when applying the operator $|\nabla|^{s_c}$.
	\end{remark}
	\begin{proof}
		By direct computation, we have
		\begin{align}
			&\nabla [|x|^{-b}(F(u+v)-F(u))]\notag\\
			&= |x|^{-b-1}[F(u+v)-F(u)] + |x|^{-b}[F'(u+v)-F'(u)] \nabla u  + |x|^{-b}F'(u+v)\nabla v \notag \\
			&:= A_1(t,x) + A_2(t,x) + A_3(t,x). \notag
		\end{align}
		
		We first estimate $A_1(t,x)$.  Note that  $\rho=\frac{4(\alpha +1)}{2\alpha +1}$ satisfies 
	  $			\frac{3}{4}=\alpha (\frac{1}{\rho}-\frac{s_c}{4})+\frac{1}{\rho}+\frac{b}{4}.$
		Using Lemma \ref{L:leibnitz}, Lemma \ref{L:6141}, and Sobolev embedding, we obtain
		\begin{align}
		&	\||\nabla|^{s_c-1} A_1(t,x)\|_{L_t^2 L_x^{\frac{4}{3},2}(I \times \mathbb{R}^4)} \notag\\
			&\lesssim \|F(u+v)-F(u)\|_{L_t^2 L_x^{l_1,2}(I\times\mathbb{R}^4)}\, \||x|^{-b-s_c}\|_{L_x^{\frac{4}{b+s_c},\infty}}\notag\\
			&\quad + \||\nabla|^{s_c-1}\int _0^1F'(u+\lambda v)vd\lambda\|_{L_t^2 L_x^{l_2,2}(I\times\mathbb{R}^4)}\,\||x|^{-b-1}\|_{L_x^{\frac{4}{b+1},\infty}}\notag\\
			&\lesssim M_{s_c}(u,v),\notag
		\end{align}
		where  $\frac{1}{l_1}:=\frac{3-b-s_c}{4}=(\alpha +1)(\frac{1}{\rho}-\frac{s_c}{4})$ and  $\frac{1}{l_2}:=\frac{2-b}{4}=\alpha (\frac{1}{\rho}-\frac{s_c}{4})+\frac{1}{\rho}-\frac{1}{4}$.

		Next, we consider $A_2(t,x)$. By Lemma \ref{L:leibnitz}, Lemma \ref{L:6141}, and the Sobolev embedding again, we have
		\begin{align}
			&\||\nabla|^{s_c-1} A_2(t,x)\|_{L_t^2 L_x^{\frac{4}{3},2}(I\times\mathbb{R}^4)}\notag\\
			&\lesssim \|[F'(u+v)-F'(u)] \nabla u\|_{L_t^2 L_x^{l_3,2}(I\times\mathbb{R}^4)}\,\||x|^{-b-s_c+1}\|_{L_x^{\frac{4}{b+s_c-1},\infty}}\notag\\
			&\quad + \||\nabla|^{s_c-1}\int_0^1F''(u+\lambda v)vd\lambda\|_{L_t^{\frac{\gamma}{\alpha}}L_x^{l_4}(I\times\mathbb{R}^4)}\,\||\nabla|^{s_c}u\|_{L_t^\gamma L_x^{\rho,2}(I\times\mathbb{R}^4)}\,\||x|^{-4}\|_{L_x^{\frac{4}{b},\infty}},\label{E581}
		\end{align}
		where  $\frac{1}{l_3}:=\frac{4-b-s_c}{4}=\alpha (\frac{1}{\rho}-\frac{s_c}{4})+\frac{1}{\rho}-\frac{s_c-1}{4}$ and  $\frac{1}{l_4}:=\frac{2-b+s_c}{4}-\frac{1}{\rho}=(\alpha -1)(\frac{1}{\rho}-\frac{s_c}{4})+\frac{1}{\rho}-\frac{1}{4}$.  
		
	When $\alpha \ge 2$,   $F''(u)$ is Lipschitz continuous. Therefore, by applying Lemma \ref{L:6141},  Hölder  and Sobolev embedding to estimate the right-hand side of   (\ref{E581}), we  obtain 
	\begin{equation}
	\||\nabla|^{s_c-1} A_2(t,x)\|_{L_t^2 L_x^{\frac{4}{3},2}(I\times\mathbb{R}^4)}\lesssim  M_{s_c}(u,v).\label{E582}
\end{equation}
		  In the case $1<\alpha<2$,  $F''(u)$ is  only Hölder continuous of order $\alpha - 1$.
		  Note that 
		   \begin{equation}
		  	\frac{s_c-1}{\alpha -1}<s_c \iff b<2s_c+\frac{2}{s_c}-3,\notag
		  \end{equation}
		  and   $2s_c+\frac{2}{s_c}-3>1+(s_c-1)^2$ for $s_c\in (1,2)$. Therefore, by applying Lemma \ref{LVisan},  Hölder   and Sobolev embedding to estimate (\ref{E581}), we obtain (\ref{E582}).
		
		Similarly, by  the same argument as that used to estimae  $A_2(t,x)$, we get 
		\begin{align}
			\||\nabla|^{s_c-1} A_3(t,x)\|_{L_t^2 L_x^{\frac{4}{3},2}(I\times\mathbb{R}^4)}
			\lesssim M_{s_c}(u,v).\notag
		\end{align}
		
		Combining the above estimates, we obtain the desired estimate in Lemma \ref{Lemmanonlinearestimate}.
	\end{proof}
 
	\subsection{The construction of  scattering solutions away from the origin.}\label{s:2.4}
	In this section, as explained in Introduction, to prove Theorem \ref{TReduction}, we only need to establish a key proposition about the existence of scattering solutions to (\ref{INLS}) associated to initial data living sufficiently far from the origin.
	
	\begin{proposition}\label{Pscatteringsolution}
		Let  $\lambda_n\in (0,\infty )$,  $x_n\in \mathbb{R} ^4$, satisfy 
	 $\lim_{n \rightarrow \infty}\frac{ |x_n| }{\lambda_n}= \infty$. Let $\phi \in \dot{H}^{s_c}(\mathbb{R}^4)$. Then for sufficiently large $n$, there exists a global scattering solution $u_n(t,x)$ to equation (\ref{INLS}) such that
		\begin{equation}
			u_n(0) = \lambda_n^{s_c-2}\phi(\frac{x-x_n}{\lambda_n}) \quad \text{and} \quad \||\nabla|^{s_c} u_n\|_{L_t^{\gamma}L_x^{\rho,2}(I \times \mathbb{R}^4)} \lesssim \|\phi\|_{\dot{H}^{s_c}(\mathbb{R}^4)}, \label{E328w1}
		\end{equation}
		where  $(\gamma ,\rho)=(2(\alpha +1),\frac{4(\alpha +1)}{2\alpha +1})$. 
	\end{proposition}
	
	\begin{proof}
		By scalling, it sufficies to consider the case  $\lambda_n\equiv1$. Define $v_n(t,x) := e^{it\Delta} \phi(x - x_n)$. Then $v_n(t,x)$ solves the perturbed equation
		\begin{equation}
			i\partial_t v_n(t,x) + \Delta v_n(t,x) = |x|^{-b} |v_n(t,x)|^{\alpha} v_n(t,x) +e_n(t,x), \notag
		\end{equation}
		where  $e_n(t,x):=- |x|^{-b} |v_n(t,x)|^{\alpha} v_n(t,x) $. 
		
		By Strichartz, it is easy to see that  $e_n(t,x)$ is bounded in  $L_t^{\infty }\dot H^{s_c}_x(I\times \mathbb{R} ^4)\cap L_t^{\gamma }\dot W^{s_c}L_x^{\rho,2}(I\times \mathbb{R} ^4)$. To apply the stability theorem \ref{TStability}, it suffices to show that
		\begin{equation}
			\lim_{n \rightarrow \infty} \||\nabla|^{s_c}e_n(t,x)\|_{L_t^2 L_x^{\frac{4}{3},2}(I \times \mathbb{R}^4)} = 0. \label{E3283}
		\end{equation}
		Once this is established, the existence of the desired solution $u_n(t,x)$ satisfying (\ref{E328w1}) follows directly from the stability theorem \ref{TStability}.
		
		We now proceed to prove estimate (\ref{E3283}). For any $\varepsilon > 0$, we first choose a function $\psi(t,x) \in C_c^\infty(\mathbb{R} \times \mathbb{R}^4)$ such that
		\begin{equation}
			\||\nabla|^{s_c}[\psi(t,x) - e^{it\Delta} \phi(x)]\|_{L_t^\gamma L_x^{\rho,2}(I \times \mathbb{R}^4)} < \varepsilon. \notag
		\end{equation}
		
		Fix $T > 0$. We first estimate $e_n$ on the time interval $[-T, T]$.  Let 
		\begin{equation}
			\Phi_n(t,x) := |x|^{-b} |\psi(t,x-x_n)|^\alpha \psi(t,x-x_n). \notag
		\end{equation}
		Then, by Hölder, Strichartz and Sobolev embedding, we  have   
		\begin{align}
			&\||\nabla|^{s_c}[e_n(t,x) - \Phi_n(t,x)]\|_{L_t^2 L_x^{\frac{4}{3},2}( [-T, T] \times \mathbb{R}^4)} \notag \\
			&\lesssim \left( \||\nabla|^{s_c} e^{it\Delta} \phi\|_{L_t^\gamma L_x^{\rho,2}}^\alpha + \||\nabla|^{s_c} \psi(t,x)\|_{L_t^\gamma L_x^{\rho,2}}^\alpha \right) \||\nabla|^{s_c}[\psi(t,x) - e^{it\Delta} \phi(x)]\|_{L_t^\gamma L_x^{\rho,2}} 
			\lesssim \varepsilon.\label{E583}
		\end{align}
		On the other hand, since the function $\psi(t,x)$ has compact support in space and $|x_n|\rightarrow\infty$ as $n\rightarrow\infty$, it follows that
		\begin{equation}
			\lim_{n\rightarrow\infty} \||\nabla|^{s_c}\Phi_n(t,x)\|_{L_t^{2}L_x^{\frac{4}{3},2}([-T,T]\times \mathbb{R}^4)} = 0.\label{E584}
		\end{equation}
		Combining  (\ref{E583}) and (\ref{E584}), we obtain 
		\begin{equation}
			\limsup _{n\rightarrow\infty }\||\nabla|^{s_c}e_n(t,x)\|_{L_t^2 L_x^{\frac{4}{3},2}([-T,T]\times \mathbb{R}^4)} \lesssim \varepsilon.  \label{E328w2}
		\end{equation}

		We now estimate $e_n$ on the set $\{t : |t| > T\}$. Again by Hölder’s inequality, we have
		\begin{equation}
			\||\nabla|^{s_c} e_n\|_{L_t^2 L_x^{\frac{4}{3},2}(\{|t| > T\} \times \mathbb{R}^4)} \lesssim \||\nabla|^{s_c}e^{it\Delta} \phi\|_{L_t^\gamma L_x^{\rho,2}(\{|t| > T\} \times \mathbb{R}^4)}^{\alpha + 1}, \label{E585}
		\end{equation}
		which tends to zero as $T \to \infty$, by the Strichartz estimate and the monotone convergence theorem. 
		The desired estimate (\ref{E3283}) then follows from  
		 (\ref{E328w2})  and (\ref{E585}). This  completes the proof of Proposition \ref{Pscatteringsolution}.
	\end{proof}

	\section{The case  $1<s_c<\frac{3}{2}$}\label{S4}
	In this section, we rule out the existence of almost periodic solutions as described in Theorem \ref{TReduction} for the case $1 < s_c < 3/2$. The argument relies on a space-localized Morawetz inequality, similar to the method employed by Bourgain \cite{Bourgain1999} in the study of the radial energy-critical NLS.
	\begin{lemma}[Morawetz inequality]\label{L1091}
		Let $1 < s_c < \frac{3}{2}$, and let $u$ be a solution to (\ref{INLS}) on the time interval $I$. Then, for any $A \geq 1$, the following estimate holds:
		\begin{equation}
			\int_I \int_{|x| \leq A |I|^{1/2}} \frac{|u(t,x)|^{\alpha + 2}}{|x|^{1+b}} \, dx \, dt \lesssim (A|I|^{\frac{1}{2}})^{2s_c-1} \left\{ \|u\|_{L_t^{\infty} \dot H_x^{s_c}(I \times \mathbb{R}^4)}^2 + \|u\|_{L_t^{\infty} \dot H_x^{s_c}(I \times \mathbb{R}^4)}^{\alpha + 2} \right\}. \label{E1092}
		\end{equation}
	\end{lemma}
	\begin{proof}
		Let $\phi(x)$ be a smooth, radial bump function such that $\phi(x) = 1$ for $|x| \leq 1$ and $\phi(x) = 0$ for $|x| > 2$. Let  $R>1$ and $a(x) := |x| \phi\left(\frac{x}{R}\right)$. 
		
		Define the Morawetz potential as
		\begin{equation}
			M(t) := 2\text{Im} \int \partial_j a(x) \overline{u} (t,x)\partial_j u(t,x) \, dx, \notag
		\end{equation}
		where we sum the repeated indices. Using equation (\ref{INLS}), direct computation gives
		\begin{align}
			\frac{d}{dt}M(t) &= 4\text{Re} \int \partial_j \partial_k a \partial_j \overline{u} \partial_k u \, dx - \int |u|^2 \Delta^2 a \, dx \notag \\
			&\qquad + \frac{2\alpha}{\alpha + 2} \int |x|^{-b} |u|^{\alpha + 2} \Delta a \, dx + \frac{4b}{\alpha + 2} \int |x|^{-b-2} |u|^{\alpha + 2} x_j \partial_j a \, dx. \notag
		\end{align}
		
		For $|x| \leq R$, note that $\partial_j \partial_k a$ is positive definite, and $\Delta a(x) = 3/|x|$. This implies
		\begin{align}
			\frac{d}{dt}M(t) &\geq \frac{6\alpha + 4b}{\alpha + 2} \int_{|x| <R} |x|^{-b-1} |u|^{\alpha + 2} \, dx+4\text{Re} \int_{|x| > R} \partial_j \partial_k a\partial_j \overline{u} \partial_k u \, dx \notag\\
			&\quad- \int_{|x| > R} |u|^2 \Delta^2 a \, dx + \frac{2\alpha }{\alpha + 2} \int_{|x| > R} |x|^{-b} |u|^{\alpha + 2} \Delta a \, dx \notag\\
			&\quad+ \frac{4b}{\alpha + 2} \int_{|x| > R} |x|^{-b-2} |u|^{\alpha + 2} x_j \partial_j a \, dx. \label{E12220}
		\end{align}
		
		Using Hölder's inequality and Sobolev embedding, we estimate
		\begin{align}
			&\int_{|x| > R} |\partial_j \partial_k a \partial_j \overline{u} \partial_k u| \, dx + \int_{|x| > R} |u|^2 |\Delta^2 a| \, dx \notag \\
			&\lesssim \|\Delta a\|_{L_x^{\frac{2}{s_c-1}}(|x| > R)} \|\nabla u\|_{L_x^{\frac{4}{3-s_c}}(\mathbb{R}^4)}^2 + \|u\|_{L_x^{\frac{4}{2-s_c}}(\mathbb{R}^4)}^2 \|\Delta^2 a\|_{L_x^{\frac{2}{s_c}}(\mathbb{R}^4)} \notag \\
			&\lesssim R^{2s_c-3} \|u\|_{\dot H^{s_c}(\mathbb{R}^4)}^2, \label{E12221}
		\end{align}
		and
		\begin{align}
			&\int_{|x| > R} |x|^{-b} |u|^{\alpha + 2} |\Delta a| \, dx + \int_{|x| > R} |x|^{-b-2} |u|^{\alpha + 2} |x_j \partial_j a| \, dx \notag \\
			&\lesssim R^{-b} \|\Delta a\|_{L_x^{\frac{4}{2(s_c-1)+b}}(|x| > R)} \|u\|_{L_x^{\frac{4}{2-s_c}}(\mathbb{R}^4)}^{\alpha + 2} + R^{-b-1} \|\nabla a\|_{L_x^{\frac{4}{2(s_c-1)+b}}(|x|>R)} \|u\|_{L_x^{\frac{4}{2-s_c}}(\mathbb{R}^4)}^{\alpha + 2} \notag \\
			&\lesssim R^{2s_c-3} \|u\|_{L_x^{\frac{4}{2-s_c}}(\mathbb{R}^4)}^{\alpha + 2}\lesssim R^{2s_c-3}\|u\|_{\dot H^{s_c}(\mathbb{R} ^4)}^{\alpha +2}. \label{E12222}
		\end{align}
		
		On the other hand, Hölder's inequality together with Sobolev embedding yields the following bound for $M(t)$:
		\begin{equation}
			|M(t)| \lesssim \|u\|_{L_x^{\frac{4}{2-s_c}}(\mathbb{R}^4)} \|\nabla u\|_{L_x^{\frac{4}{3-s_c}}(\mathbb{R}^4)} \|\nabla a\|_{L_x^{\frac{4}{2s_c-1}}(\mathbb{R}^4)} \lesssim R^{2s_c-1} \|u\|_{\dot H^{s_c}(\mathbb{R} ^4)}^2. \label{E12223}
		\end{equation}
		Substituting (\ref{E12221})--(\ref{E12223}) into (\ref{E12220}) and integrating over the time interval $I$, we obtain
		\begin{equation}
			\int _I\int_{|x| \leq R} \frac{|u(t,x)|^{\alpha + 2}}{|x|^{1+b}} \, dx \, dt \lesssim R^{2s_c-1} \|u\|_{L_t^{\infty} \dot H_x^{s_c}(I \times \mathbb{R}^4)}^2 + R^{2s_c-3} |I| \|u\|_{L_t^{\infty} \dot H_x^{s_c}(I \times \mathbb{R}^4)}^{\alpha + 2}. \notag
		\end{equation}
		Taking $R = A |I|^{\frac{1}{2}}$ gives the desired estimate (\ref{E1092}).
	\end{proof}
	
	Now, we use the spatially localized Morawetz inequality (Lemma \ref{L1091}) to rule out the existence of the  almost periodic solutions in Theorem \ref{TReduction} with $1<s_c<\frac{3}{2}$.   
	
	\begin{theorem}\label{T1-3}
		Let  $1<s_c<\frac{3}{2}$. There are no global radial almost periodic solutions in Theorem \ref{TReduction} satisfying $T_{\text{max}} = +\infty$.
	\end{theorem}
	\begin{proof}
		We argue by contradiction. Assume that $u$ is a  almost periodic solution in Theorem \ref{TReduction} and satisfies $T_{\text{max}} = +\infty$.   
		For the almost periodic solution $u$, we can use the compactness property (\ref{Ecompact}) to deduce the following inequality (a detailed proof can be found in \cite[(7.4)]{MiaoMurphyZheng2014}):  
		\begin{equation}
			\inf_{t \in [0, +\infty)} N(t)^{2s_c} \int_{|x| \leq \frac{C(u)}{N(t)}} |u(t,x)|^2 \, dx \gtrsim_u 1. \label{Emass}
		\end{equation}
		Let $I \subset [0, \infty)$ be the union of contiguous characteristic subspaces $J_k$. Choose $C(u)$ sufficiently large. Using the assumption (\ref{Ebound}), the spatially localized Morawetz inequality (Lemma \ref{L1091}), Hölder's inequality, and   (\ref{Emass}), we obtain:
		\begin{align*}
			|I|^{s_c - \frac{1}{2}} 
			&\gtrsim_u \sum_{J_k \subset I} \int_{J_k} \int_{|x| \leq C(u)|J_k|^{\frac{1}{2}}} \frac{|u(t,x)|^{\alpha + 2}}{|x|^{1+b}} \, dx \, dt \\
			&\gtrsim_u \sum_{J_k \subset I} N_k^{1+b} \int_{J_k} \int_{|x| \leq C(u) N_k^{-1}} |u(t,x)|^{\alpha + 2} \, dx \, dt \\
			&\gtrsim_u \sum_{J_k \subset I} \int_{J_k}  N_k^{1 + b + 2\alpha} \left( \int_{|x| \leq C(u)N_k^{-1}} |u(t,x)|^2 \, dx \right)^{\frac{\alpha + 2}{2}} \, dt \\
			&\gtrsim_u \sum_{J_k \subset I} \int_{J_k} N_k^{1 + b + 2\alpha} \left( N(t)^{-2s_c} \right)^{\frac{\alpha + 2}{2}} \, dt \\
			&\gtrsim_u \int_I N(t)^{3 - 2s_c} \, dt \gtrsim_u |I|,
		\end{align*}
		where, in the last inequality, we use the fact that $s_c < \frac{3}{2}$ and $N(t) \geq 1$. If we take the length of the interval $I$ to be sufficiently large in the inequalities above, we reach a contradiction. Thus, Theorem \ref{T1-3} is proved.
	\end{proof}

	\section{The case  $\frac{3}{2}\le s_c<2$: Rapid frequency-cascade}\label{S5}
	This section is devoted to ruling out the rapid frequency-cascade solutions as described in Theorem \ref{TReduction2} for the case $\frac{3}{2} \leq s_c < 2$.   The main technical  tools used are the long-time Strichartz estimate and a frequency-localized Morawetz estimate.
	
	\subsection{Long-time Strichartz estimates}\label{S51}
Here, we establish long-time Strichartz estimates tailored to the Lin-Strauss Morawetz inequality. These estimates were originally introduced by Dodson \cite{Dodson2012} in the context of the mass-critical NLS. See also \cite{LiuMiaoZheng,LuZheng2017,Murphy2015} for further developments that are more closely related to our setting.
		\begin{theorem}[Long-time Strichartz estimate]\label{Tlong time SZ}
		Let $s_c \ge \frac{3}{2}$, and let $u : [0, T_{\text{max}}) \times \mathbb{R}^4 \to \mathbb{C}$ be an almost periodic solution to \eqref{INLS} with $N(t) \equiv N_k \geq 1$ on each characteristic subinterval $J_k \subset [0, T_{\text{max}})$. Assume
		\[
		u \in L_t^\infty([0, T_{\text{max}}), \dot{H}_x^s(\mathbb{R}^4)), \quad \text{for some } s_c - \frac{1}{2} < s \leq s_c.  
		\]
		Then, on any compact time interval $I \subset [0, T_{\text{max}})$, which is a union of characteristic subintervals $J_k$, and for any $N > 0$, the following holds:
		\begin{equation}
			\| |\nabla|^s u_{\leq N} \|_{L_t^2 L_x^{4,2}(I \times \mathbb{R}^4)} \lesssim_u 1 + N^{\sigma(s)} K_I^{1/2}, \label{3.2}
		\end{equation}
		where $K_I := \int_I N(t)^{3-2s_c} \, dt$ and $\sigma(s) := 2s_c - s - \frac{1}{2}$. Specifically, for $s = s_c$, we have
		\begin{equation}
			\| |\nabla|^{s_c} u_{\leq N} \|_{L_t^2 L_x^{4,2}(I \times \mathbb{R}^4)} \lesssim_u 1 + N^{s_c - \frac{1}{2}} K_I^{1/2}. \label{3.3}
		\end{equation}
		Moreover, for any $\eta > 0$, there exists $N_0 = N_0(\eta)$ such that for all $N \leq N_0$,
		\begin{equation}
			\| |\nabla|^{s_c} u_{\leq N} \|_{L_t^2 L_x^{4,2}(I \times \mathbb{R}^4)} \lesssim_u \eta (1 + N^{s_c - \frac{1}{2}} K_I^{1/2}). \label{3.4}
		\end{equation}
		The implicit constants in \eqref{3.2} and \eqref{3.4} are independent of $I$.
	\end{theorem}

By Lemma \ref{LSpacetime bounds}, it is straightforward to verify that Theorem \ref{Tlong time SZ} holds for sufficiently large $N$. For general $N > 0$, we proceed by induction. Following the argument in the proof of [Proposition 4.1, \cite{Murphy2015}] or [Theorem 3.1, \cite{LuZheng2017}], it suffices to prove the following iterative lemma:
\begin{lemma}[Iterative formula of $B(N)$]\label{Iter}
	Let $B(N):=\||\nabla|^{s_c}u_{\leq N}\|_{L_t^2L_x^{4,2}(I\times\mathbb{R}^4)}$. For any $\epsilon,\ \epsilon_0>0$, there exists a constant $C(\varepsilon,\varepsilon_0)>0$ such that 
	\begin{align}\label{Iter}
		B(N)\lesssim_u&\inf_{t\in I}\||\nabla|^{s_c}u_{\leq N}(t)\|_{L_x^2}+C(\varepsilon,\varepsilon_0)N^{\sigma(s)}K_I^{\frac12}
		\notag\\
		&\quad+\varepsilon^\alpha B(N/\varepsilon_0)+\sum_{M>N/\varepsilon_0}\Big(\frac{N}{M}\Big)^{s_c}B(M)
	\end{align}
	holds uniformly in $N$. 
\end{lemma}

\begin{proof}
	First, by Strichartz estimate, we have 
\begin{equation}
			B(N)\lesssim \inf_{t\in I}\||\nabla|^{s_c}u_{\leq N}(t)\|_{L_x^2}+\||\nabla|^{s_c}P_{\leq N}|x|^{-b}F(u)\|_{L_t^2L_x^{\frac43,2}([0,T_{\max})\times\mathbb{R}^4)}.\label{E452}
\end{equation}
	To estimate the nonlinear term, we decompose
	$$F(u)=F(u_{\leq N/\epsilon_0})+(F(u)-F(u_{\leq N/\epsilon_0})).$$
	Using Bernstein's inequality, Hölder's inequality, Sobolev embedding, and the bound (\ref{Ebound}), we obtain
	\begin{align}
		&\||\nabla|^{s_c}P_{\leq N}|x|^{-b}(F(u)-F(u_{\leq N/\epsilon_0}))\|_{L_t^2L_x^{\frac43,2}(I\times\mathbb{R}^4)} \notag\\
		\lesssim &N^{s_c}\||x|^{-b}\|_{L_x^{\frac{4}b,\infty}(\mathbb{R} ^4)}\|(F(u)-F(u_{\leq N/\epsilon_0}))\|_{L_t^2L_x^{\frac4{3-b},2}(I\times\mathbb{R}^4)} \notag\\
		\lesssim& N^{s_c}\|u\|_{L_t^\infty L_x^{\frac{4\alpha}{2-b}}(I\times\mathbb{R}^4)}^\alpha\|u_{>N/\epsilon_0}\|_{L_t^2L_x^{4,2}(I\times\mathbb{R}^4)}\notag\\
		\lesssim&\sum_{M>N/\epsilon_0}\Big(\frac{N}{M}\Big)^{s_c}B(M).\label{E453}
	\end{align}
We now turn to estimating $F(u_{\le N/\varepsilon _0})$. To begin, we focus on a single characteristic interval $J_k$. By applying Lemma \ref{L:leibnitz}, the Sobolev embedding, and the bound (\ref{Ebound}), we obtain the following estimate
	\begin{align*}
		&\||\nabla|^{s_c}P_{\leq N}|x|^{-b}F(u_{\leq N/\epsilon_0})\|_{L_t^2L_x^{\frac43,2}(J_k\times\mathbb{R}^4)}\\
		\lesssim&\||x|^{-b}\|_{L_x^{\frac4b,\infty}(\mathbb{R} ^4)}\||\nabla|^{s_c}F(u_{\leq N/\epsilon_0})\|_{L_t^2L_x^{\frac4{3-b},2}(J_k\times\mathbb{R}^4)}\\
		&+\||x|^{-b-s_c}\|_{L_x^{\frac{4}{b+s_c},\infty}(\mathbb{R} ^4)}\|F(u_{\leq N/\epsilon_0})\|_{L_t^2L_x^{\frac4{3-b-s_c},2}(J_k\times\mathbb{R}^4)}\\
		\lesssim&\|u_{\leq N/\epsilon_0}\|_{L_t^\infty L_x^{\frac{4\alpha}{2-b}}(I\times\mathbb{R}^4)}^\alpha\||\nabla |^{s_c}u_{\leq N/\epsilon_0}\|_{L_t^2L_x^{4,2}(J_k\times\mathbb{R}^4)}\\
		\lesssim& \||\nabla|^{s_c}P_{\leq c(\epsilon)N_k}u_{\leq N/\epsilon_0}\|_{L_t^\infty L_x^2(I\times\mathbb{R}^4)}^\alpha \||\nabla|^{s_c}u_{\leq N/\epsilon_0}\|_{L_t^2 L_x^{4,2}(J_k\times\mathbb{R}^4)}\\
		&+\||\nabla|^{s_c}P_{>c(\epsilon)N_k}u_{\leq N/\epsilon_0}\|_{L_t^\infty L_x^2(I\times\mathbb{R}^4)}^\alpha
		\||\nabla|^{s_c}u_{\leq N/\epsilon_0}\|_{L_t^2 L_x^{4,2}(J_k\times\mathbb{R}^4)}\\
		:=&I_k+II_k.
	\end{align*}
	By the compactness property, there exists $c=c(\epsilon)$ such that $\||\nabla|^{s_c}u_{\leq c(\epsilon)N(t)}\|_{L_t^\infty L_x^2}<\epsilon$, and thus
	$$I_k\lesssim \epsilon^\alpha\||\nabla|^{s_c}u_{\leq N/\epsilon_0}\|_{L_t^2 L_x^{4,2}(J_k\times\mathbb{R}^4)}.$$
	
	For the term $II_k$, we only need to consider the case $c(\epsilon)N_k\leq N/\epsilon_0$. By Bernstein's inequality and Lemma \ref{LSpacetime bounds}, we obtain
	\begin{align*}
		II_k\lesssim& \Big(\frac{N}{\epsilon_0c(\epsilon)N_k}\Big)^{s_c-\frac12}\||\nabla|^{s_c}u_{\leq N/\epsilon_0}\|_{L_t^\infty L_x^{2}(J_k\times\mathbb{R}^4)}\\
		\lesssim& \Big(\frac{N}{\epsilon_0c(\epsilon)N_k}\Big)^{s_c-\frac12}\Big(\frac{N}{\epsilon_0}\Big)^{s_c-s}\|u_{\leq N/\epsilon_0}\|_{L_t^\infty \dot{H}^s(J_k\times\mathbb{R}^4)}.
	\end{align*}
	
	Summing over all characteristic intervals, we deduce
	\begin{align*}
		&\||\nabla|^{s_c}P_{\leq N}|x|^{-b}F(u_{\leq N/\epsilon_0})\|_{L_t^2L_x^{\frac43,2}(I\times\mathbb{R}^4)}^2\\
		\lesssim&\sum_{J_k\subset I}\Big[\epsilon^{2\alpha}\||\nabla|^{s_c}u_{\leq N/\epsilon_0}\|_{L_t^2 L_x^{4,2}(J_k\times\mathbb{R}^4)}^2\\
		&\qquad\quad+\Big(\frac{N}{\epsilon_0c(\epsilon)N_k}\Big)^{2s_c-1}\Big(\frac{N}{\epsilon_0}\Big)^{2s_c-2s}\|u_{\leq N/\epsilon_0}\|_{L_t^\infty \dot{H}^s(J_k\times\mathbb{R}^4)}^2\Big]\\
		\lesssim&\epsilon^{2\alpha}B^2(N/\epsilon_0)+c(\varepsilon )^{-2s_c+1}\epsilon_0^{-2\sigma(s)}\sup_{J_k}\|u_{\leq N/\epsilon_0}\|_{L_t^\infty\dot{H}^s(J_k\times\mathbb{R}^4)}^2N^{2\sigma}K_I,
	\end{align*}
which together with (\ref{E452})	yields (\ref{Iter}).  This completes the proof of the lemma.
\end{proof}

	\subsection{The rapid frequency-cascade scenario}
	In this subsection, we use the long-time Strichartz estimate proved in the previous section to rule out   the existence of rapid frequency-cascade solutions, i.e. the almost periodic solutions as described in Theorem \ref{TReduction2} such that 
	\begin{equation}
		\int_0^{T_{\text{max}}} N(t)^{3-2s_c} \, dt < \infty.
	\end{equation}
	\begin{theorem}[No rapid frequency-cascades]\label{Tnorapid}
		Let $s_c \ge \frac{3}{2}$. Then there are no almost periodic solutions $u : [0, T_{\text{max}}) \times \mathbb{R}^4 \to \mathbb{C}$ to \eqref{INLS} with $N(t) \equiv N_k \geq 1$ on each characteristic subinterval $J_k \subset [0, T_{\text{max}})$ such that  $u$ blows up forward in time and 
		\begin{equation}
			K:=\int _0^{T_{\text{max}}}N(t)^{3-2s_c}dt<+\infty .\label{E129}
		\end{equation}
	\end{theorem}
	
	The proof proceeds by contradiction. Suppose such a solution $u$ exists. From \eqref{E129} and Corollary 1.7, it follows that
\begin{equation}
		\lim_{t \to T_{\text{max}}} N(t) = \infty,  \label{E451}
\end{equation}
	whether $T_{\text{max}}$ is finite or infinite. Together with the compactness property, this implies
	\begin{align}\label{Gs-2}
		\lim_{t \to T_{\text{max}}} \| |\nabla|^{s_c} u_{\leq N} \|_{L_x^2(\mathbb{R}^4)} = 0,  
	\end{align}
	for any $N > 0$.

To proceed, we require the following Proposition \ref{P1} and Proposition \ref{P2}. These results show that an almost periodic solution satisfying (\ref{E129}) must exhibit negative regularity. In the final part of this subsection, we will exploit this negative regularity to derive a contradiction, thereby concluding the proof of Theorem \ref{Tnorapid}.
\begin{proposition}[Lower regularity]\label{P1}
	Let $ \frac{3}{2}\le s_c<2$ and let $u : [0, T_{\text{max}}) \times \mathbb{R}^4 \to \mathbb{C}$ be an almost periodic solution to \eqref{INLS} satisfying \eqref{E129}. Suppose that
	\begin{equation}
		u \in L_t^\infty([0, T_{\text{max}}), \dot{H}_x^s(\mathbb{R}^4)), \quad \text{for some } s_c - \frac{1}{2} < s \leq s_c.\notag
	\end{equation}
	Then it follows that
	\begin{equation}
		u \in L_t^\infty([0, T_{\text{max}}), \dot{H}_x^\beta(\mathbb{R}^4)), \quad \forall s - \sigma(s) < \beta \leq s_c,\notag
	\end{equation}
	where $\sigma(s) = 2s_c - s - \frac{1}{2}$.
\end{proposition}

 The proof of Proposition \ref{P1} follows a similar strategy to that of [Lemma 4.3, \cite{Murphy2015}] and [Proposition 4.2, \cite{LuZheng2017}], where their arguments crucially rely on the negative derivative estimates provided by Lemma \ref{Lemmaparaproduct}. However, such estimates have not yet been established in the Lorentz space setting due to the lack of a Coifman–Meyer multiplier theorem in this framework.
To circumvent this issue, we first apply embedding theorems to reduce the Lorentz norms to Lebesgue norms, and then invoke the negative derivative estimates from Lemma \ref{Lemmaparaproduct}.  
Subsequently, to handle the inhomogeneous coefficient in the Lebesgue space setting, we will make use of the nonlinear estimates  in  Lemma \ref{L-qd}.

\begin{proof}[\textbf{Proof of Proposition \ref{P1}}]
	Let $I_n\subset [0,T_{max})$ be unions of characteristic subintervals, each consisting of a finite number of contiguous characteristic subintervals. We first claim that for any $N>0$, the following estimate holds:
	\begin{align}\label{Gs-1}
		B_n(N)\lesssim_u \inf_{t\in I_n}\||\nabla|^{s_c}u_{\leq N}(t)\|_{L_x^2}+N^{\sigma(s)},
	\end{align}
	where 
	$B_n(N):=\||\nabla|^{s_c}u_{\leq N}\|_{L_t^2L_x^{4,2}(I_n\times\mathbb{R}^4)}.$
	
	In fact,  by the same argument as that used to derive (\ref{Iter}), we obtain
	\begin{align*}
		B_n(N)\lesssim_u&\inf_{t\in I_n}\||\nabla|^{s_c}u_{\leq N}(t)\|_{L_x^2}+C(\varepsilon,\varepsilon_0)N^{\sigma(s)}
		\notag\\
		&\quad+\varepsilon^\alpha B_n(N/\varepsilon_0)+\sum_{M>N/\varepsilon_0}\Big(\frac{N}{M}\Big)^{s_c}B_n(M).
	\end{align*}
The estimate (\ref{Gs-1}) then follows from a standard iterative argument.
	
	Letting $n \to \infty$ in \eqref{Gs-1} and applying \eqref{Gs-2}, we  get 
	\begin{equation}
		\||\nabla|^{s_c}u_{\leq N}\|_{L_t^2L_x^{4,2}([0,T_{\max})\times\mathbb{R}^4)}\lesssim_u N^{\sigma(s)}, \quad \forall\ N\geq0.\label{E431}
	\end{equation}
	
	In what follows, we use (\ref{E431}) to prove that
	\begin{equation}
		\||\nabla|^su_{\leq N}\|_{L_t^\infty L_x^2([0,T_{\max})\times \mathbb{R}^4)}\lesssim N^{\sigma(s)}.\label{E432}
	\end{equation}
	First, by  the no-waste Duhamel formula (Lemma \ref{LNoWasteDuhamel}) and the Strichartz estimate, we obtain
	\begin{align}
		\||\nabla|^su_{\leq N}\|_{L_t^\infty L_x^2([0,T_{\max})\times \mathbb{R}^4)}\lesssim \||\nabla|^sP_{\leq N}(|x|^{-b}F(u))\|_{L_t^2 L_x^{\frac43,2}([0,T_{\max})\times\mathbb{R}^4)}.\notag
	\end{align}
	Then we decompose $F(u)$ as
	$$F(u)=F(u_{\leq N})+(F(u)-F(u_{\leq N})).$$
	For the first term, using Bernstein, H\"older and estimate (\ref{E431}), we obtain
	\begin{align*}
	&\||\nabla|^sP_{\leq N}|x|^{-b}F(u_{\leq N})\|_{L_t^2 L_x^{\frac43,2}([0,T_{\max})\times\mathbb{R}^4)}\\
	\lesssim&\| |x|^{-b}\|_{L_x^{\frac4b,\infty}}\|u\|_{L_t^\infty L_x^{\frac{4\alpha}{s_c-s+2-b}}([0,T_{\max})\times\mathbb{R}^4)}^{\alpha}\||\nabla|^su_{\leq N}\|_{L_t^2L_x^{\frac{4}{1-s_c+s},2}([0,T_{\max})\times\mathbb{R}^4)}\\
	&+\||x|^{-b}\|_{L_x^{\frac4b,\infty}}\||x|^{-s}F(u_{\leq N})\|_{L_t^2L_x^{\frac4{3-b},2}([0,T_{\max})\times\mathbb{R}^4)}\\
	\lesssim& \|u\|_{L_t^\infty\dot{H}^{s_1}}^{\alpha}\||\nabla|^{s_c}u_{\leq N}\|_{L_t^2L_x^{4,2}([0,T_{\max})\times\mathbb{R}^4)}+\||\nabla|^{s}F(u_{\leq N})\|_{L_t^2L_x^{\frac4{3-b},2}([0,T_{\max})\times\mathbb{R}^4)}\\
	\lesssim& \|u\|_{L_t^\infty\dot{H}^{s_1}}^{\alpha}\||\nabla|^{s_c}u_{\leq N}\|_{L_t^2L_x^{4,2}([0,T_{\max})\times\mathbb{R}^4)}\\
	\lesssim& N^{\sigma(s)},
\end{align*}
where $s_1:=2-\frac{s_c-s+2-b}{\alpha}=s_c-\frac{s_c-s}{\alpha }\in (s,s_c]$, and in the second inequality we used   the Hardy's inequality and the  embedding  $\dot H^{s_1}(\mathbb{R} ^4)\hookrightarrow  L^{\frac{4\alpha }{s_c-s+2-b}}(\mathbb{R} ^4)$.

For the second term, we consider two cases. When $s = s_c$, by arguing as in (\ref{E453}) and applying (\ref{E431}), we obtain  
\begin{align}
	& \||\nabla|^sP_{\leq N}|x|^{-b}(F(u)-F(u_{\leq N}))\|_{L_t^2 L_x^{\frac43,2}([0,T_{\max})\times\mathbb{R}^4)}\notag\\
	&\lesssim \sum _{M>N}(\frac{N}{M})^{s_c} \||\nabla|^{s_c}u_{\le M}\|_{L_t^{2}L_x^{4,2}([0,T_{\max})\times\mathbb{R}^4)}\notag\\
	&\lesssim  \sum _{M>N}(\frac{N}{M})^{s_c}  M^{\sigma(s)} \lesssim N^{\sigma (s)}.\notag
\end{align}
When $s < s_c$, we have  $s_1=s_c-\frac{s_c-s}{\alpha }\in (s,s_c)$. By employing the Bernstein inequality, the embedding $L_x^{\frac{4}{3}} \hookrightarrow L_x^{\frac{4}{3},2}$, Lemma \ref{Lemmaparaproduct}, Lemma \ref{L-qd}, and the Sobolev embedding theorem, we obtain
\begin{align*}
	& \||\nabla|^sP_{\leq N}|x|^{-b}(F(u)-F(u_{\leq N}))\|_{L_t^2 L_x^{\frac43,2}([0,T_{\max})\times\mathbb{R}^4)}\\
	\lesssim& N^{s_c}\left\||\nabla|^{s-s_c}\int_{0}^{1}|x|^{-b}F'(u_{\leq N}+\lambda u_{\geq N})u_{\geq N}d{\lambda}\right\|_{L_t^2 L_x^{\frac43}}\\
	\lesssim& N^{s_c}\int_{0}^1\||\nabla|^{s_c-s}|x|^{-b}F'(u_{\leq N}+\lambda u_{\geq N})\|_{L_t^\infty L_x^{\frac{2}{s_c-s+1}}}d\lambda \||\nabla|^{s-s_c}u_{\geq N}\|_{L_t^2 L_x^{\frac{4}{1-s_c+s}}}\\
	\lesssim&N^{s_c} \||\nabla|^{s_c-s}u\|_{L_t^\infty L_x^{\frac{4\alpha}{(\alpha+1)(s_c-s)+2-b}}}\left(\|u\|_{L_t^\infty L_x^{\frac{4\alpha}{s_c-s+2-b }-}}^{\alpha-1}\|u\|_{L_t^\infty L_x^{\frac{4\alpha}{s_c-s+2-b }+}}^{\alpha-1}\right)^\frac12\|u_{>N}\|_{L_t^2L_x^{4,2}}\\
	\lesssim&\|u\|_{L_t^\infty\dot{H}^{s_1}}\left(\|u\|_{L_t^\infty \dot{H}^{s_1^-}}^{\alpha-1}\|u\|_{L_t^\infty \dot{H}^{s_1^+}}^{\alpha-1}\right)^\frac12\sum_{M>N}\Big(\frac{N}{M}\Big)^{s_c}\||\nabla|^{s_c}u_M\|_{L_t^2L_x^{4,2}}\\
	\lesssim& N^{\sigma(s)},
\end{align*}
In the final two steps, we have used the Sobolev embeddings $\dot H^s (\mathbb{R} ^4)\cap \dot H^{s_c}(\mathbb{R} ^4)\hookrightarrow \dot H^{s_1\pm}(\mathbb{R}^4) \hookrightarrow L^{\frac{4\alpha }{s_c - s + 2 - b} \pm}(\mathbb{R}^4)$ along with the bound (\ref{E431}). This completes the proof of (\ref{E432}).

Finally, we apply (\ref{E432}) to complete the proof of Proposition \ref{P1}. By the Bernstein inequality and (\ref{E432}), we obtain:
	\begin{align*}
		\||\nabla|^{\beta }u\|_{L_t^\infty L_x^2} &\lesssim \||\nabla|^{\beta}u_{\geq1}\|_{L_t^\infty L_x^2}+\sum_{N\leq 1}N^{\beta-s}\||\nabla|^{s}u_N\|_{L_t^\infty L_x^2}\\
		&\lesssim \||\nabla|^{s_c}u_{\geq1}\|_{L_t^\infty L_x^2}+\sum_{N\leq 1}N^{\beta-s+\sigma(s)}\lesssim 1,
	\end{align*}
	provided that $s-\sigma(s)<\beta\le s_c$. This complets the proof of Proposition \ref{P1}. 
\end{proof}

\begin{proposition}[Negative regularity]\label{P2}
	Let $u : [0, T_{\text{max}}) \times \mathbb{R}^4 \to \mathbb{C}$ be an almost periodic solution to \eqref{INLS} as in Theorem \ref{TReduction2} with \eqref{E129}. Suppose that
	\[
	\inf_{t \in [0, T_{\text{max}})} N(t) \geq 1.
	\]
	Then, for any $0 < \varepsilon < \frac{1}{2}$, we have
\begin{equation}
		u \in L_t^\infty([0, T_{\text{max}}); \dot{H}_x^{-\varepsilon}(\mathbb{R}^4)).  \label{E43w2}
\end{equation}
\end{proposition}

\begin{proof}
	First, applying Proposition \ref{P1} with $s_c = s$, we obtain
	$$u\in L_t^\infty([0,T_{\text{max}});\dot{H}_x^{\beta}), \quad \text{for all } \frac12<\beta\leq s_c.$$
	Then, applying Proposition \ref{P1} again with $s=(s_c - \frac{1}{2})_{+}$ yields   (\ref{E43w2}).  
\end{proof}

\begin{proof}[\textbf{Proof of Theorem \ref{Tnorapid}}]
	 Since  $	\left\{ N(t)^{s_c-2} u(t, N(t)^{-1} x) : t \in [0, T_{\max}) \right\}$
	is precompact in $\dot{H}^{s_c}(\mathbb{R}^4)$, by Arzala-Ascoli theorem, for any  $\eta>0, $ there exists  $c(\eta)>0$ such that 
	\begin{equation}
		\int _{|\xi|\le c(\eta)N(t)}|\xi|^{2s_c}|\hat u(t.\xi)|^2d\xi \le \eta \label{E591}
	\end{equation}
	Interpolating between the inequalities (\ref{E591}) and (\ref{E43w2}), we obtain that for any $\eta > 0$, there holds:
	$$\int_{|\xi|\leq c(\eta)N(t)}|\hat{u}(t,\xi)|^2d\xi\lesssim_u\eta^{\frac{\epsilon}{s_c+\epsilon}}.$$
	Therefore, using (\ref{Gs-2}), we obtain
	\begin{align*}
		M[u_0]=M[u(t)] &= \int_{|\xi|\leq c(\eta)N(t)}|\hat{u}(t,\xi)|^2d\xi + \int_{|\xi|\geq c(\eta)N(t)}|\hat{u}(t,\xi)|^2d\xi \\
		&\lesssim_u \eta^{\frac{\epsilon}{s_c+\epsilon}} + (c(\eta)N(t))^{-2s_c}\||\nabla|^{s_c}u\|_{L_t^\infty L_x^2}^2\\
		&\lesssim_u \eta^{\frac{\epsilon}{s_c+\epsilon}} + (c(\eta)N(t))^{-2s_c}.
	\end{align*}
	Letting $t \to T_{\text{max}}$, by (\ref{E451}) we deduce $M(u_0) = 0$, hence $u_0 = 0$, which contradicts the fact that  $u$ blows up forward in time. This completes the proof of Theorem \ref{Tnorapid}.
\end{proof}

	\section{The Case $\frac{3}{2} \leq s_c < 2$: Soliton}\label{S6}
	In this section, we rule out  soliton solutions as described in Theorem \ref{TReduction2}  for the case $\frac{3}{2} \leq s_c < 2$. Subsection \ref{S61} focuses on establishing a frequency-localized Lin-Strauss Morawetz inequality, which will be utilized in Subsection \ref{S62} to rule out soliton solutions.
	
	\subsection{The Frequency-Localized Morawetz Inequality}\label{S61}
	This subsection develops spacetime estimates for the high-frequency components of almost periodic solutions to \eqref{INLS}. These estimates will be applied in the next section to preclude the existence of quasi-soliton solutions as described in Theorem \ref{TReduction2}.
	
	\begin{theorem}[Frequency-Localized Morawetz Inequality]\label{TMorawetz}
		Let $s_c \ge  \frac{3}{2}$, and let $u : [0, T_{\text{max}}) \times \mathbb{R}^4 \to \mathbb{C}$ be an almost periodic solution to \eqref{INLS} with $N(t) \equiv N_k \geq 1$ on each characteristic subinterval $J_k \subset [0, T_{\text{max}})$. Let $I \subset [0, T_{\text{max}})$ be a compact interval composed of consecutive characteristic subintervals $J_k$. Then, for any $\eta > 0$, there exists $N_0 = N_0(\eta)$ such that for all $N \leq N_0$, the following holds:
		\[
		\int_I \int_{\mathbb{R}^4} \frac{|u_{>N}(t,x)|^{\alpha +2}}{|x|} \, dx \, dt \lesssim_u \eta \big(N^{1-2s_c} + K_I\big), \tag{5.1}
		\]
		where $K_I := \int_I N(t)^{3-2s_c} \, dt$. Moreover, $N_0$ and the implicit constants in the inequality are independent of the interval $I$.
	\end{theorem}
 
We need the following lemma to prove Theorem \ref{TMorawetz}.
	\begin{lemma}[Control of Low and High Frequencies]\label{LControl of Low and High Frequencies}
		Let $u : [0, T_{\text{max}}) \times \mathbb{R}^4 \to \mathbb{C}$ be an almost periodic solution to \eqref{INLS} with $N(t) \equiv N_k \geq 1$ on each characteristic subinterval $J_k \subset [0, T_{\text{max}})$. Then, on any compact interval $I \subset [0, T_{\text{max}})$, composed of consecutive subintervals $J_k$, the following estimates hold:\\
		(i) Let  $(q,r)$ be an admissible pair. For any  $N>0$ and  $0\le s<1/2$, we have 
		\begin{equation}
			\||\nabla |^{s}u_{\geq N}\|_{L_t^q L_x^{r,2}(I \times \mathbb{R}^4)} \lesssim_u N^{s-s_c} \left( 1 + N^{2s_c - 1} K_I \right)^{\frac{1}{q}}.  \label{E12241}
		\end{equation}
		\noindent (ii)
		For any $\eta > 0$, there exists $N_1 = N_1(\eta)$ such that for all $N \leq N_1$, we have
		\[
		\||\nabla|^{\frac{1}{2}} u_{\geq N}\|_{L_t^\infty L_x^2(I \times \mathbb{R}^4)} \lesssim_u \eta N^{\frac{1}{2} - s_c}.  
		\]
		\noindent (iii) Using \eqref{3.4}, it follows that for any $\eta > 0$, there exists $N_2 = N_2(\eta)$ such that for all $N \leq N_2$,
		\begin{equation}
			\||\nabla|^{s_c} u_{\leq N}\|_{L_t^2 L_x^{4,2}(I \times \mathbb{R}^4)} \lesssim_u \eta \left( 1 + N^{2s_c - 1} K_I \right)^{\frac{1}{2}}.  \label{E12242}
		\end{equation}
	\end{lemma}   
	\begin{proof}
The proof follows from the same method  outlined in  \cite[Lemma 4.7]{Murphy2015} and \cite[Lemma 5.2]{LuZheng2017}, and we omit the details here for brevity.
	\end{proof}

	\noindent \textbf{Proof of Theorem \ref{TMorawetz}}
	Throughout the proof, the spacetime norms are taken over $I \times \mathbb{R}^4$.
	
	We set $0 < \eta \ll 1$ and choose  
	\[
	N < \min \left\{ N_1(\eta), \eta^2 N_2(\eta^{2s_c}) \right\},
	\]
	where $N_1$ and $N_2$ are constants determined in Lemma \ref{LControl of Low and High Frequencies}. Notably, from (\ref{E12241}), we have
	\begin{equation}
		\|u_{>N/\eta^2}\|_{L_t^{2}L_x^{4,2}} \lesssim_u \eta N^{-s_c}(1 + N^{2s_c-1}K_I)^{1/2}. \label{E12243}
	\end{equation}
	Additionally, since $N/\eta^2 < N_2(\eta^{2s_c})$, we can apply (\ref{E12242}) to obtain
	\begin{equation}
		\||\nabla|^{s_c}u_{\leq N/\eta^2}\|_{L_t^{2}L_x^{4,2}} \lesssim_u \eta (1 + N^{2s_c-1}K_I)^{1/2}. \label{E12244}
	\end{equation}
	
	We define the Morawetz functional as 
	\[
	M(t): = 2 \text{Im} \int_{\mathbb{R}^4} \frac{x}{|x|} \cdot \nabla u_{>N}(t,x) \overline{u}_{>N}(t,x) dx. \notag
	\]
	By computing using the equation $(i\partial_t + \Delta) u_{>N} = P_{>N} F(u)$, we have
	\[
	\partial_t M \gtrsim \int_{\mathbb{R}^4} \frac{x}{|x|} \cdot \{P_{>N} F(u), u_{>N}\}_p dx, \notag
	\]
	where the momentum bracket $\{\cdot, \cdot\}_p$ is defined as $\{f, g\}_p := \text{Re}(f \nabla \overline{g} - g \nabla \overline{f})$. Thus, by the fundamental theorem of calculus, we obtain
	\begin{equation}
		\int_I \int_{\mathbb{R}^4} \frac{x}{|x|} \cdot \{P_{>N} F(u), u_{>N}\}_p dx dt \lesssim \|M(t)\|_{L_t^{\infty}(I)}. \label{E12245}
	\end{equation}
	Observing that $\{F(u), u\}_p = -\frac{\alpha}{\alpha + 2} \nabla (|x|^{-b} |u|^{\alpha + 2}) - \frac{2}{\alpha + 2} \nabla (|x|^{-b}) |u|^{\alpha + 2}$, we can write
	\begin{align}
		\{P&_{>N}F(u),u_{>N}\}_p\notag\\
		&=\{F(u),u\}_p-\{F(u_{\le N},u_{\le N})\}_p  -\{F(u)-F(u_{\le N}),u_{\le N}\}_p-\{P_{\le N}F(u),u_{>N}\}_p\notag\\
		&=-\frac{\alpha }{\alpha +2}\nabla [|x|^{-b}(|u|^{\alpha +2}-|u_{\le N}|^{\alpha +2})]-\frac{2}{\alpha +2}\nabla (|x|^{-b})(|u|^{\alpha +2}-|u_{\le N}|^{\alpha +2})\notag\\
		&\qquad -\{F(u)-F(u_{\le N}),u_{\le N}\}_p-\{P_{\le N}F(u),u_{>N}\}_p\notag\\
		&:=I+II+III.\notag
	\end{align}
	
	For $I$, integrating by parts shows that it contributes to the LHS of (\ref{E12245}) a multiple of 
	\[
	\int_I \int_{\mathbb{R}^4} \frac{|u_{>N}(t,x)|^{\alpha + 2}}{|x|^{1+b}} dx dt \notag
	\]
	and to the RHS of (\ref{E12245})  a multiple of
	\begin{align}
		\|\frac{1}{|x|^{1+b}}&(|u|^{\alpha + 2} - |u_{>N}|^{\alpha + 2} - |u_{\leq N}|^{\alpha + 2})\|_{L_{t,x}^1} \notag\\
		&\lesssim \|\frac{1}{|x|^{1+b}}(u_{\leq N})^{\alpha + 1}u_{>N}\|_{L_{t,x}^1} + \|\frac{1}{|x|^{1+b}}(u_{>N})^{\alpha + 1}u_{\leq N}\|_{L_{t,x}^1} := I_1 + I_2. \label{EM1}
	\end{align}
	
	For $II$, integration by parts yields
\begin{align}
		II &\lesssim \|\frac{1}{|x|^{1+b}} u_{\leq N} (|u|^\alpha u - |u_{\leq N}|^\alpha u_{\leq N})\|_{L_{t,x}^1} + \|\frac{1}{|x|^b} \nabla u_{\leq N} (|u|^\alpha u - |u_{\leq N}|^\alpha u_{\leq N})\|_{L_{t,x}^1} \notag\\
		&:= II_1 + II_2. \label{EM2}
\end{align}
	
	Similarly, for $III$, we have
	\begin{equation}
		III \lesssim \|\frac{1}{|x|^{1+b}} u_{>N} P_{\leq N} (|u|^\alpha u)\|_{L_{t,x}^1} + \|\frac{1}{|x|^b} u_{>N} \nabla P_{\leq N} (|u|^\alpha u)\|_{L_{t,x}^1}  . \label{EM3}
	\end{equation}
	
	Thus, to complete the proof of Theorem \ref{TMorawetz}, it suffices to show that  
	\begin{equation}
		\|M(t)\|_{L_t^{\infty }(I)} \lesssim_u \eta N^{1-2s_c}, \label{EM4}
	\end{equation}
	and that the error terms (\ref{EM1}), (\ref{EM2}), and (\ref{EM3}) are all controlled by $\eta (N^{1-2s_c} + K_I)$.
	
	We begin by proving (\ref{EM4}). Using Bernstein's inequality and (\ref{Ebound}), we obtain 
	\begin{align}
		\|M(t)\|_{L_t^{\infty }(I)} 
		&\lesssim \||\nabla|^{-1/2} \nabla u_{>N}\|_{L_t^{\infty} L_x^{2}} \||\nabla|^{1/2} \left( \frac{x}{|x|} u_{>N} \right)\|_{L_t^{\infty} L_x^{2}} \notag\\
		&\lesssim \||\nabla|^{1/2} u_{>N}\|_{L_t^{\infty} L_x^{2}}^2 \lesssim \eta N^{1-2s_c}. \notag
	\end{align}
	
	Next, we estimate the error terms (\ref{EM1}), (\ref{EM2}), and (\ref{EM3}).
	
	First, by interpolation, (\ref{Ebound}), and Lemma \ref{LControl of Low and High Frequencies}, we have  
	\begin{equation}
		\||\nabla|^{s_c}u_{\leq N}\|_{L_t^{2(\alpha+1)}L_x^{\frac{4(\alpha+1)}{2\alpha+1},2}}^{\alpha+1} 
		\lesssim \||\nabla|^{s_c}u\|_{L_t^{\infty}L_x^{2}}^{\alpha} \||\nabla|^{s_c}u_{\leq N}\|_{L_t^{2}L_x^{4,2}} 
		\lesssim \eta (1+N^{2s_c-1}K_I)^{\frac{1}{2}}. \notag
	\end{equation}
	Combining the above inequality with Hölder's inequality, Hardy's inequality, Bernstein's inequality, Sobolev embedding, and Lemma \ref{LControl of Low and High Frequencies}, we obtain  
	\begin{align}
		I_1 
		&\lesssim \||x|^{-\frac{1+b}{\alpha+1}}u_{\leq N}\|_{L_t^{2(\alpha+1)}L_x^{\frac{4(\alpha+1)}{3},2}}^{\alpha+1} \|u_{>N}\|_{L_t^{2}L_x^{4,2}} \notag\\
		&\lesssim \||\nabla|^{\frac{1+b}{\alpha+1}}u_{\leq N}\|_{L_t^{2(\alpha+1)}L_x^{\frac{4(\alpha+1)}{3},2}}^{\alpha+1} \|u_{>N}\|_{L_t^{2}L_x^{4,2}} \notag\\
		&\lesssim N^{1-s_c} \||\nabla|^{s_c}u_{\leq N}\|_{L_t^{2(\alpha+1)}L_x^{\frac{4(\alpha+1)}{2\alpha+1},2}}^{\alpha+1} \|u_{>N}\|_{L_t^{2}L_x^{4,2}} \notag\\
		&\lesssim \eta N^{1-2s_c}(1+N^{2s_c-1}K_I). \notag
	\end{align}
	For the estimate of $I_2$, we consider two cases. 
	If $|u_{\leq N}| \ll |u_{>N}|$, then this term can be absorbed into the LHS of (\ref{E12245}), provided we can show that  
	\begin{equation}
		\|\frac{1}{|x|^{1+b}}|u_{>N}|^{\alpha +2}\|_{L_{t,x}^{1}} < \infty. \label{EM5}
	\end{equation}
	For the other cases of $I_2$, we reduce the problem to estimating $I_1$. Therefore, to complete the estimate of $I_2$, it suffices to prove (\ref{EM5}). For this, using Hardy's inequality, Sobolev embedding, Bernstein's inequality, and Lemma \ref{LSpacetime bounds}, we have  
	\begin{align}
		\|\frac{1}{|x|^{1+b}}|u_{>N}|^{\alpha+2}\|_{L_{t,x}^1} 
		&\lesssim \||x|^{-\frac{1+b}{\alpha+2}}u_{>N}\|_{L_{t,x}^{\alpha+2}}^{\alpha+2} \lesssim \||\nabla|^{\frac{1+b}{\alpha+2}}u_{>N}\|_{L_t^{\alpha+2}L_x^{\alpha+2}}^{\alpha+2} \notag\\
		&\lesssim \||\nabla|^{\frac{2\alpha-1-b}{\alpha+2}}u_{>N}\|_{L_t^{\alpha+2}L_x^{\frac{2(\alpha+2)}{\alpha+1},2}} 
		\notag\\
		&\lesssim N^{1-2s_c} \||\nabla|^{s_c}u\|_{L_t^{\alpha+2}L_x^{\frac{2(\alpha+2)}{\alpha+1},2}} < \infty,\notag
	\end{align}
	where we also used the embedding  $ L^{\frac{2(\alpha +2)}{\alpha +1},2}(\mathbb{R} ^4)\hookrightarrow L^{\frac{2(\alpha +2)}{\alpha +1}}(\mathbb{R} ^4)$.
	
	Next, we consider the estimate of (\ref{EM2}). The term $II_1$ can be directly controlled by $I_1 + I_2$. For $II_2$, using Hölder's inequality, Hardy's inequality, Sobolev embedding, (\ref{Ebound}), and Lemma \ref{LControl of Low and High Frequencies}, we have  
	\begin{align}
		II_2 
		&\lesssim \|\frac{1}{|x|^{b}} \nabla u_{\leq N} (|u_{>N}|^{\alpha} + |u_{<N}|^{\alpha}) u_{>N}\|_{L_{t,x}^1} \notag\\
		&\lesssim \left( \||x|^{-\frac{b}{\alpha}} u_{>N}\|_{L_t^{\infty}L_x^{2\alpha}}^\alpha + \||x|^{-\frac{b}{\alpha}} u_{>N}\|_{L_t^{\infty}L_x^{2\alpha}}^\alpha \right) \|u_{>N}\|_{L_t^{2}L_x^{4,2}} \|\nabla u_{\leq N}\|_{L_t^{2}L_x^{4,2}} \notag\\
		&\lesssim \||\nabla|^{\frac{b}{\alpha}} u\|^{\alpha }_{L_t^{\infty}L_x^{2\alpha}} \|u_{>N}\|_{L_t^{2}L_x^{4,2}} \|\nabla u_{\leq N}\|_{L_t^{2}L_x^{4,2}} \lesssim \||\nabla|^{s_c} u\|_{L_t^{\infty}L_x^{2}}^\alpha \|u_{>N}\|_{L_t^{2}L_x^{4,2}} \|\nabla u_{\leq N}\|_{L_t^{2}L_x^{4,2}} \notag\\
		&\lesssim \eta N^{1-2s_c}(1+N^{2s_c-1}K_I). \notag
	\end{align}
	
	Finally, we consider the estimate of (\ref{EM3}). First, using H\"older and Hardy's inequality, we have  
	\begin{align}
		III 
		&\lesssim \|u_{>N}\|_{L_t^{2}L_x^{4,2}} \|\frac{1}{|x|^{1+b}} P_{\leq N}(|u|^{\alpha} u)\|_{L_t^{2}L_x^{\frac{4}{3},2}} 
		+ \|u_{>N}\|_{L_t^{2}L_x^{4,2}} \|\frac{1}{|x|^{b}} \nabla P_{\leq N}(|u|^{\alpha} u)\|_{L_t^{2}L_x^{\frac{4}{3},2}} \notag\\
		&\lesssim\|u_{>N}\|_{L_t^{2}L_x^{4,2}}  \||x|^{-b}\|_{L_x^{\frac{4}{b},\infty }} \|\nabla P_{\leq N}(|u|^{\alpha} u)\|_{L_t^{2}L_x^{\frac{4}{3-b},2}}. \notag
	\end{align}
	Thus, in light of (\ref{E12241}), it suffices to prove  
	\begin{equation}
		\|\nabla P_{\leq N}(|u|^{\alpha} u)\|_{L_t^{2}L_x^{\frac{4}{3-b},2}} \lesssim_u \eta N^{1-s_c}(1+N^{2s_c-1}K_I)^{1/2}. \notag
	\end{equation}
	
	To prove the above estimate, we use Hölder's inequality, Bernstein's inequality, the fractional chain rule, (\ref{Ebound}), (\ref{E12243}), and (\ref{E12244}) to estimate
	\begin{align}
		&\|\nabla P_{\leq N}(|u|^{\alpha} u)\|_{L_t^{2}L_x^{\frac{4}{3-b},2}} \notag\\
		&\lesssim \|\nabla P_{\leq N}(|u|^{\alpha} u - |u_{\leq N/\eta^2}|^\alpha u_{\leq N/\eta^2})\|_{L_t^{2}L_x^{\frac{4}{3-b},2}} 
		+ N^{1-s_c} \||\nabla|^{s_c}(|u_{\leq N/\eta^2}|^\alpha u_{\leq N/\eta^2})\|_{L_t^{2}L_x^{\frac{4}{3-b},2}} \notag\\
		&\lesssim N \|u\|_{L_t^{\infty}L_x^{\frac{4\alpha}{2-b},2}}^\alpha \|u_{>N/\eta^2}\|_{L_t^{2}L_x^{4,2}} 
		+ N^{1-s_c} \|u\|_{L_t^{\infty}L_x^{\frac{4\alpha}{2-b},2}}^\alpha \||\nabla|^{s_c} u_{\leq N/\eta^2}\|_{L_t^{2}L_x^{4,2}} \notag\\
		&\lesssim_u \eta N^{1-s_c}(1+N^{2s_c-1}K_I)^{1/2}. \notag
	\end{align}
	
	This completes the proof of Theorem \ref{TMorawetz}.
	\subsection{Ruling Out the Soliton}\label{S62}
	In this subsection, we rule out  the soliton solutions as in  Theorem \ref{TReduction2}. 
	\begin{theorem}[No quasi-solitons]\label{Tnosoliton}
		Let  $\frac{3}{2}\le s_c<2$.	There are no almost periodic solutions as in Theorem \ref{TReduction2} such that 
		\begin{equation}
			K_{[0,T_{\text{max}})}=\int _0^{T_{\text{max}}}N(t)^{3-2s_c}dt=\infty .\label{1224z1}
		\end{equation}
	\end{theorem}
	The proof of Theorem \ref{TReduction2} relies on the frequency-localized Morawetz inequality established in Subsection \ref{S61} and the following lower bound for the Morawetz estimate:
	
	\begin{lemma}[Lower bound]\label{Llowerbound}  
		Let $\frac{3}{2} \leq s_c < 2$, and let $u:[0, T_{\text{max}}) \times \mathbb{R}^4 \rightarrow \mathbb{C}$ be an almost periodic solution as in Theorem \ref{TReduction2}. Then, there exists $N_0 > 0$ such that for any $N < N_0$, we have  
		\begin{equation}
			K_I \lesssim_u \int_I \int_{\mathbb{R}^4} \frac{|u_{>N}(t,x)|^{\alpha+2}}{|x|^{1+b}} dx dt, \label{E1224z2}
		\end{equation}  
		where $K_I := \int_I N(t)^{3-2s_c} dt$.
	\end{lemma}
	
	\begin{proof}  
		First, by the same argument as in \cite[(7.3)]{MiaoMurphyZheng2014}, there exist a sufficiently large constant $C(u)$ and $N_0 > 0$ such that for all $N < N_0$, we have  
		\begin{equation}
			\inf_{t \in I} N(t)^{2s_c} \int_{|x| \leq \frac{C(u)}{N(t)}} |u_{>N}(t,x)|^2 dx \gtrsim_u 1. \notag
		\end{equation}  
		
		This inequality, combined with Hölder's inequality, directly yields (\ref{E1224z2}):  
		\begin{align}
			\int_I \int_{\mathbb{R}^4} \frac{|u_{>N}(t,x)|^{\alpha+2}}{|x|^{1+b}} dx dt  
			&\gtrsim_u \int_I N(t)^{1+b} \int_{|x| \leq \frac{C(u)}{N(t)}} |u_{>N}|^{\alpha+2} dx dt \notag\\
			&\gtrsim_u \int_I N(t)^{1+b+2\alpha} \left( \int_{|x| \leq \frac{C(u)}{N(t)}} |u_{>N}|^2 dx \right)^{\frac{\alpha+2}{2}} dt \notag\\
			&\gtrsim_u \int_I N(t)^{1+b+2\alpha} \left(N(t)^{-2s_c}\right)^{\frac{\alpha+2}{2}} dt \notag\\
			&= \int_I N(t)^{3-2s_c} dt=K_I. \notag
		\end{align}  
		This completes the proof of Lemma \ref{Llowerbound}. 
	\end{proof}
	Now, we return to the proof of Theorem \ref{Tnosoliton}. Assume that $u$ is an almost periodic solution satisfying (\ref{1224z1}). Combining Theorem \ref{TMorawetz} and Lemma \ref{Llowerbound}, we obtain  
	\begin{equation}
		K_I \lesssim_u \eta (N^{1-2s_c} + K_I). \notag
	\end{equation}  
	By choosing $\eta > 0$ sufficiently small, we deduce that $K_I \lesssim_u N^{1-2s_c}$ holds uniformly over the interval $I$. Taking $I \subset [0, T_{\text{max}})$ to be sufficiently large leads to a contradiction. This completes the proof of Theorem \ref{Tnosoliton}. 
	Therefore, we conclude Theorem \ref{T1}.

\end{document}